\documentclass[11pt, oneside]{amsart} 
\usepackage{mathtools, amsmath, amsthm, amssymb, calrsfs, wasysym, verbatim, bbm, color, graphics, geometry, enumerate, hyperref}

\usepackage{mathrsfs}
\usepackage{dsfont}
\usepackage{setspace}
\usepackage{graphicx}
\usepackage{mathtools}
\usepackage{array}
\usepackage{stmaryrd}
\usepackage{esvect}
\usepackage{tikz}
\usepackage{xcolor}
\usepackage{scalerel}
\usepackage{tikz-cd}
\usepackage{quiver}
\usepackage[T1]{fontenc}
\usepackage{cleveref}

\DeclareMathAlphabet{\mathcal}{OMS}{cmsy}{m}{n}

\newcommand{\R}{\mathbb{R}}
\newcommand{\C}{\mathbb{C}}
\newcommand{\Z}{\mathbb{Z}}

\newcommand{\Q}{\mathbb{Q}}
\newcommand{\G}{\mathbb{G}}

\def\R{\mathbb{R}}
\def\Z{\mathbb{Z}}
\def\Q{\mathbb{Q}}

\def\C{\mathbb{C}}

\def\F{\mathbb{F}}

\def\fld{\mathrm{f{\kern 1pt}ld\,}}
\def\sub{\subseteq}

\def\inv{^{-1}}

\def\mc{\mathcal}

\def\dashedrightarrow{\mathrel{\ThisStyle{\ooalign{$\SavedStyle\rightarrow$\cr%
    \hfil\textcolor{white}{\rule{2\LMpt}{1\LMex}}\kern2\LMpt\hfil}}}}

\def\Gal{\mathrm{Gal}}

\def\eps{\varepsilon}
\def\ol{\overline}
\def\GL{\mathrm{GL}}
\makeatletter
    \def\mmod#1{\allowbreak\mkern10mu({\operator@font mod}\,\,#1)}
\makeatother

\def\mf{\mathfrak}

\DeclareMathAlphabet{\mathdutchcal}{U}{dutchcal}{m}{n}
\SetMathAlphabet{\mathdutchcal}{bold}{U}{dutchcal}{b}{n}
\DeclareMathAlphabet{\mathdutchbcal}{U}{dutchcal}{b}{n}

\def\mf{\mathfrak}

\def\xto{\xrightarrow}
\def\Hom{\mathrm{Hom}}

\def\End{\mathrm{End}}
\def\sups{\supseteq}

\def\Spec{\operatorname{Spec}}

\def\gr{\operatorname{gr}}

\def\iso{\cong}

\def\dirlim{\varinjlim}

\def\onto{\twoheadrightarrow}
\def\into{\hookrightarrow}

\def\SL{\mathrm{SL}}
\def\GL{\mathrm{GL}}

\def\rank{\mathrm{rank}}
\def\O{\mc O}
\def\tr{\mathrm{tr}}

\def\Der{\mathrm{Der}}

\def\st{\,|\,}
\def\Biggst{\,\Bigg|\,}

\def\Bigst{\,\Big|\,}
\def\biggst{\,\bigg|\,}
\def\symb{\mathrm{symb}}
\def\SU{\mathrm{SU}}
\def\sc{_\mathrm{sc}}
\def\der{_\mathrm{der}}
\def\ad{_\mathrm{ad}}
\def\Lie{\mathrm{Lie}}
\def\aug{\epsilon}
\def\an{^\mathrm{an}}
\def\alg{\mathrm{alg}}
\def\WR{\mathrm{R}}

\newcommand{\genlegendre}[4]{%
  \genfrac{(}{)}{}{#1}{#3}{#4}%
  \if\relax\detokenize{#2}\relax\else_{\!#2}\fi
}

\theoremstyle{plain}
\newtheorem{thm}{Theorem}[section]
\newtheorem{lem}[thm]{Lemma}
\newtheorem{prop}[thm]{Proposition}
\newtheorem{cor}[thm]{Corollary}
\newtheorem{claim}[thm]{Claim}

\theoremstyle{definition}
\newtheorem{defn}[thm]{Definition}
\newtheorem{ex}[thm]{Example}
\newtheorem{cons}[thm]{Construction}

\theoremstyle{remark}
\newtheorem{rem}[thm]{Remark}

\newtheorem{note}[thm]{Note}
\newtheorem{notn}[thm]{Notation}

\hypersetup{
    pdftitle={On the saturability of p-adic Lie groups},
    pdfauthor={Casey Perdue}
}

\title{On the saturability of $p$-adic Lie groups}
\author{Casey Perdue}
\date{}

\numberwithin{equation}{section}

\begin{document}

\begin{abstract}
Lazard theory is becoming increasingly important in $p$-adic representation theory, especially when studying mod $p$ representations.  For $G$ a connected reductive group defined over a finite extension $K/\mathbb Q_p$, a question going back to Lazard is when a pro-$p$ Iwahori subgroup of $G(K)$ admits a saturated $p$-valuation.  In this article, we establish a deep connection to Moy-Prasad theory, and show that the Moy-Prasad filtration defines a saturated $p$-valuation on $G_{x,\frac{1}{p-1}+}$ for any point $x$ in the semisimple Bruhat-Tits building of $G$, given a mild assumption on the maximal torus of $G$ and that $p>2$.  We also compute Lazard Lie algebras defined over $\mathbb Z_p$ and the polynomial ring $\mathbb F_p[\varepsilon]$ associated to the saturated group $G_{x,\frac{1}{p-1}+}$, which allows us to reinterpret the Moy-Prasad isomorphism as an isomorphism of Lie algebras over $\F_p[\eps]$.
\end{abstract}
\maketitle
\tableofcontents


\section{Introduction}

\subsection{Lazard theory}

Since its introduction in Lazard's 1965 paper \cite{lazard65}, the notion of a \textit{$p$-valuation} on a pro-$p$ group and the stronger notion of a \textit{saturated} $p$-valuation have been incredibly useful tools for studying $p$-adic Lie groups.  For example, Lazard proved in \cite{lazard65} that a topological group is a $p$-adic Lie group if and only if it contains an open subgroup which is a pro-$p$ group admitting a saturated $p$-valuation.  A more modern treatment of these same ideas (which is also in English) is presented in Part B of the recent text \cite{schneiderbook}, where this fact is essentially given as \cite[Theorem 27.1, Theorem 29.8]{schneiderbook}; a very closely related fact can be found in \cite[Theorem 8.18]{DDMS}.  

Later, Lubotzky and Mann introduced and studied the related notion of a \textit{powerful} pro-$p$ group.  In the excellent text \cite{DDMS}, its authors study \textit{uniformly powerful} pro-$p$ groups, also called \textit{uniform} pro-$p$ groups.  Many of the results Lazard found for saturated pro-$p$ groups are easier to prove for uniform pro-$p$ groups, with the proofs involving less analytic machinery, and stronger theorems are often true.  It is noted in \cite[p. 81]{DDMS} that uniform pro-$p$ groups automatically admit a saturated $p$-valuation, i.e., are \textit{saturable}, however they state that it is unclear whether or not all saturable groups are uniform pro-$p$ groups.  While this claim is false, this uncertainty and the ease of working with uniform pro-$p$ groups led to the study of $p$-valuations going out of favor for a period of time.  In fact, it is eventually noted in \cite[Proposition A.1]{HKN}, attributed to Weigel, that being a uniform pro-$p$ group is equivalent to admitting an \textit{equi-$p$-valuation}.

In recent decades, the study of saturable pro-$p$ groups has made a comeback, especially in representation theory.  The basis for much of this work originates from \cite[3.2.7]{lazard65}, where Lazard proves that the standard pro-$p$ Iwahori subgroup $I$ of $\GL_n(\mc O_K)$ and $\SL_n(\mc O_K)$ is saturable as long as $p-1>en$, where $K/\Q_p$ is a finite extension and $e=e(K/\Q_p)$ is its absolute ramification degree.  Lazard also notes that this bound on $p$ is sharp: there are examples with $p-1=en$ for which $I$ admits no $p$-valuation, let alone a saturated one.  In \cite{klopsch}, Klopsch points out the misconception present in \cite{DDMS}, and definitely proves that there are saturable groups which are not uniform by demonstrating that these pro-$p$ Iwahori subgroups are essentially \textit{never} uniform ($n\geq 2$ is sufficient to ensure this), despite them being saturable when $p$ is sufficiently large.

In \cite[Theorem 12.1]{totaro}, Totaro takes Lazard's argument further and shows that every pro-$p$ subgroup of $G(K)$ for $G$ a semisimple algebraic group defined over $K$ admits a $p$-valuation when $p>>0$ (with an explicit lower bound for $p$), however he does not prove that his $p$-valuation is saturated on a full pro-$p$ Iwahori subgroup.  Having a saturated $p$-valuation is much more useful than merely having a $p$-valuation; for example, naturally associated to a saturated pro-$p$ group is a Lie $\Z_p$-algebra (see \cite[Section 31]{schneiderbook}).  More geometrically, \cite[Proposition 4.1]{LS} shows that there is a natural rigid analytic group over $\Q_p$ associated to any saturated pro-$p$ group, which is latter promoted to a dagger group over $\Q_p$ in \cite[Proposition 5.4]{LSS}.

The issue of saturation is addressed in a slightly different context in \cite[Proposition 3.5]{LS}.  Concretely, they show the following.
\begin{thm}[Lahiri-Sorensen] \label{LS-split}
Let $G$ be a split connected reductive $K$-group for $K/\Q_p$ a finite extension, $h$ its Coxeter number, and $I\sub G(K)$ a pro-$p$ Iwahori subgroup.  If $p-1>eh$, then there exists a saturated $p$-valuation $\omega$ on $I$.
\end{thm}
While this addresses the saturation question, it requires that the reductive group $G$ be split.  To prove this, they construct a particular $p$-valuation, prove that its restriction to each term in the Iwahori factorization
\begin{equation}\label{Iwa-fac}
\prod_{\alpha\in \Phi^-}U_{\alpha,0+}\times T_{0+}\times \prod_{\alpha\in \Phi^+} U_{\alpha,0}\xrightarrow{\sim} I
\end{equation}
is saturated, and prove that this is sufficient to conclude the saturation on $I$.  It is noted in \cite[Remark 3.5]{dottolehung} that this $p$-valuation arises from a particular Moy-Prasad filtration, which we discuss in more depth below in Subsection \ref{MP-filts}.

Many recent applications of the $p$-valuation machinery to mod $p$ representation theory arise from \cite{claus-HH}.  Let now $k$ be a field of characteristic $p$.  Given a $p$-valued pro-$p$ group $G$, let $\mf g$ be its Lazard Lie $\F_p$-algebra, which is finite dimensional if $G$ is of finite rank (see \cite[Section 25]{schneiderbook} for details).  Then \cite[Theorem 5.5]{claus-HH} constructs a multiplicative spectral sequence
\begin{equation}\label{ss}
E_1^{s,t}=H^{s,t}(\mathfrak g,k)\implies H^{s+t}(G,k),
\end{equation}
where the terms on the $E_1$ page consist of graded Lie algebra cohomology, and which is converging to the continuous group cohomology of $G$.  (The result in \cite{claus-HH} applies to more general coefficients; here, we only discuss the trivial representation $k$ for simplicity.)  In combination with Theorem \ref{LS-split}, which provides $p$-valuations in certain cases for pro-$p$ groups of interest to representation theorists, this spectral sequence has been used in many recent preprints and papers, including \cite{kongsgaard}, \cite{koziol-pepin}, and \cite{dottolehung}.

\subsection{Moy-Prasad filtrations} \label{MP-filts}

Let $G$ be a connected reductive group defined over a $p$-adic field $K$ (i.e., a finite extension $K/\Q_p$), $S\leq G$ a maximal $K$-split subtorus, $\Phi=\Phi( G, S)$ its root system.  We will always assume throughout the paper that every $p$-adic field is endowed with a valuation $v$ normalized by $v(p)=1$.  Let $x\in \mathcal A(  S)$ be any fixed point in the apartment associated to $S$ in the (semisimple) Bruhat-Tits building $\mathcal B(  G)$.  Then associated to any concave function $f:\Phi\cup\{0\}\to \mathbb R$ is a compact open subgroup $G_{x,f}\leq G(K)$.  (For definitions of all of these terms, see the recent text \cite{KP}.  However, we caution the reader that \cite{KP} often assumes that the valuation on $K$ satisfies $v(K)=\Z$, which is not the case for us, so many of our references will have implicit renormalizations.)  These were originally considered by Bruhat and Tits in the case when $f(0)=0$ in their papers \cite{BT1} and \cite{BT2}, and were later generalized beyond this requirement by Moy and Prasad in \cite{MP1} and \cite{MP2}.  An important special case which is allowed once we remove the requirement that $f(0)=0$ is when $f\geq 0$ is constant, i.e., there is some $r\geq 0$ such that $f=f_r\equiv r$ always outputs $r$.  In this case, we write $G_{x,r}\coloneqq G_{x,f_r}$, and the resulting filtration on $G_{x,0}$ is called the \textit{Moy-Prasad filtration}, with the $G_{x,r}$ being called the \textit{Moy-Prasad filtration groups}.  The construction of the Moy-Prasad filtration is later slightly altered by Yu in \cite{yu-smooth_models} via the introduction of the \textit{minimal congruent filtration}, after which more desirable properties can be proven when the maximal torus $  T$ of $  G$ is not \textit{weakly induced}.

The Moy-Prasad filtration associated to a point $x\in \mathcal B(  G)$ has become an incredibly important tool for the study of the representation theory of $G(K)$, especially in the case when one is considering complex representations.  For example, in \cite{MP2}, Moy and Prasad introduce the notion of the \textit{depth} of a complex representation, defined using Moy-Prasad filtrations.  Later, \cite{adler} and \cite{yu-tame} critically use this notion to construct many supercuspidal representations of $G(K)$, which was shown in \cite{fintzen} to produce all supercuspidal representations, so long as the group splits over a tame extension and $p$ does not divide the order of the Weyl group.

\vspace{0.1in}

\subsection{Main results}

\subsubsection{\texorpdfstring{$p$}{p}-valuations on Moy-Prasad filtration groups}

We assume for technical reasons for the remainder of the paper that $p>2$ is an odd prime.  As in the previous subsection, let $G$ be a connected reductive group defined over our finite extension $K/\Q_p$, $T\leq G$ a maximally $K$-split maximal $K$-torus, and fix any point $x\in \mc B(G)$ in its (semisimple) Bruhat-Tits building.  Write $\O_K$ for the ring of integers of $K$ and $e_K$ for its absolute ramification degree (with analogous notation for any $p$-adic field).  To fix notation, for $r\geq 0$, write $G_{x,r+}=\bigcup_{s>r}G_{x,s}$, and define 
$$\omega_x(g)\coloneqq \inf\{r> 0\st g\in G_{x,r}\}$$
for $g\in G_{x,0+}$.  (The jumps of the Moy-Prasad filtration occur on a discrete set, so one can actually replace this infimum by a minimum.)

As noted earlier, \cite[Remark 3.5]{dottolehung} points out that the $p$-valuation defined on the pro-$p$ Iwahori subgroup $I$ by Lahiri-Sorensen in the split case in Theorem \ref{LS-split} is actually equal to $\omega_x$ for a particular $x\in \mathcal B(G)$ (notably, for $x$ the barycenter of the dominant base chamber), where now $G_{x,0+}=I$.  This poses an immediate natural question: do saturated $p$-valuations on $G_{x,0+}$ arise in this way for other points $x\in \mc B(G)$?  Unfortunately, one can easily find counterexamples to this question; a requirement in the definition of a $p$-valuation would need that $\omega_x(g)>\frac1{p-1}$ for all $g\in G_{x,0+}$, which would force $G_{x,0+}=G_{x,\frac1{p-1}+}$ (which is false for easy examples).  However, if one is willing to restrict their view to the smaller group $G_{x,\frac1{p-1}+}$, we show that the answer is `yes'.  Moreover, we are able to drop the requirement that $G$ is split, so long as we add a technical condition to the torus $T$ (which is very often satisfied), and prove the following in Theorem \ref{general-existence}:
\begin{thm} \label{intro-existence}
Let $G$ be a connected reductive $K$-group, and assume that $T$ is weakly induced.  Then for any $x\in \mathcal B( G)$, $\omega_x$ is a saturated $p$-valuation on $G_{x,\frac1{p-1}+}$.
\end{thm}

We note that $T$ is weakly induced if $G$ splits over a tame extension (this condition on $T$ is also satisfied if $G$ is quasisplit, semisimple, and either adjoint or simply connected), and it turns out that the condition $p-1>eh$ present in Theorem \ref{LS-split} is ultimately merely there to ensure that $G_{x,0+}=G_{x,\frac1{p-1}+}$ for $x$ the barycenter of the dominant base chamber, making this a strict generalization of the result in \cite{LS}.

We also note for posterity that one could prove abstractly that there exists some saturated $p$-valuation on $G_{x,0+}$ if $p>>0$ using a combination of Lahiri and Sorensen's Theorem \ref{LS-split} and \cite[Theorem A]{GS-K}, however this has several drawbacks: it is not especially conducive to performing computations with, it provides no information about the case when $p$ is not quite large, and it would provide us with no information about the Moy-Prasad filtrations themselves.

One natural question arises from Theorem \ref{intro-existence}: when does this result construct a saturated $p$-valuation on the full pro-$p$ Iwahori?  In other words, when is there some point $x$ in a chamber such that $G_{x,0+}=G_{x,\frac1{p-1}+}$?  As remarked above, if $p-1>eh$ and $G$ is split, then this condition is satisfied by the barycenter of the dominant base chamber.

Let now $S\leq T$ be a maximal $K$-split subtorus and $x_0\in \mc A(S)$ be an absolutely special vertex associated to a Chevalley-Steinberg system in the apartment associated to $S$, so every $x\in \mc A(S)$ can be written as $x=x_0+v$ for a unique $v\in V(S\ad)=X_*(S\ad)\otimes_\Z\R$.  If $G$ is quasisplit, we answer the above question in Corollary \ref{polytope} as the following, where $I$ denotes the pro-$p$ Iwahori subgroup associated to the dominant base chamber:
\begin{thm} \label{intro-polytope}
If $G$ is quasisplit, $T$ is weakly induced, and $T(K)_{0+}=T(K)_{\frac1{p-1}+}$, then there is an open convex polytope $U$ of the dominant base chamber $C\sub \mc A(S)$ given by
$$U=\left\{x\in \mc A(S)\Biggst x=x_0+v, \quad \frac1{p-1}<\langle \alpha, v\rangle<\frac1{m_\alpha}\left(\frac1{e_{K_{\tilde\alpha}}}-\frac1{p-1}\right)\quad \forall \alpha\in \Phi^{+,\mathrm{nd}}\right\}$$
such that $G_{x,0+}=G_{x,\frac1{p-1}+}$ for each point $x\in U$. In other words, we get a distinct saturated $p$-valuation $\omega_x$ on $I$ for each $x\in U$, which is described by the Moy-Prasad filtration at the point $x$.  This subset $U$ is nonempty precisely when 
$$p-1>(m_\alpha \mathrm{ht}(\alpha)+1)\cdot e_{K_{\tilde\alpha}}$$
for all $\alpha\in \Phi^{+,\mathrm{nd}}$.  
\end{thm}
Here, $m_\alpha=1$ if $\alpha$ is of type R1, and is 2 if $\alpha$ is of type R2.  (See Subsection \ref{qsplit-review} for what this means, as well as precise definitions of $\Phi^{\mathrm{nd}}$ and $e_{K_{\tilde\alpha}}$.)  In comparison, by \cite[Fact 6.4.5]{KP}, the dominant base chamber $C$ itself can be described in this manner as
$$C=\left\{x\in \mc A(S)\Biggst x=x_0+v,\quad 0<\langle \alpha, v\rangle <\frac1{m_\alpha e_{K_{\tilde\alpha}}}\quad \forall \alpha\in \Phi^{+,\mathrm{nd}}\right\}.$$
If $G$ is both quasisplit and semisimple, then we are able to use ideas from \cite{totaro} to remove some of the hypotheses from Theorem \ref{intro-polytope}.  In particular, we prove the following in Corollary \ref{qsplit-ss}:
\begin{thm} 
Suppose that $G$ is a connected quasisplit semisimple $K$-group, and assume that 
$$p-1>\big(m_\alpha\mathrm{ht}(\alpha)+1\big)\cdot e_{K_{\tilde\alpha}}$$
for all $\alpha\in \Phi^{+,\mathrm{nd}}$.  Then there is an open convex nonempty polytope $U$ of $C\sub \mc A(S)$ given by 
$$U=\left\{x\in \mc A(S)\Biggst x=x_0+v, \quad \frac1{p-1}<\langle \alpha, v\rangle<\frac1{m_\alpha}\left(\frac1{e_{K_{\tilde\alpha}}}-\frac1{p-1}\right)\quad \forall \alpha\in \Phi^{+,\mathrm{nd}}\right\}$$
such that each point $x\in U$ corresponds to a distinct saturated $p$-valuation $\omega_x$ on $I$ given by the Moy-Prasad filtration at the point $x$.
\end{thm}

\subsubsection{The Lazard Lie algebras}

Via Lazard theory \cite{lazard65}, there are several Lie algebras naturally associated to a saturated pro-$p$ group.  Notably, if $(I,\omega)$ is a saturated pro-$p$ group, we have a filtered Lie $\Z_p$-algebra $L_\omega(I)_{\frac1{p-1}+}$ and filtration-preserving, mutually inverse logarithm and exponent maps
\[\begin{tikzcd}
	{L_\omega(I)_{\frac1{p-1}+}} & I
	\arrow["\exp", shift left=2, from=1-1, to=1-2]
	\arrow["\log", shift left=2, from=1-2, to=1-1]
\end{tikzcd}\]
by \cite[Section 31]{schneiderbook}.  We also have a Lie $\F_p[\eps]$-algebra 
$$\gr I=\bigoplus_{r>\frac1{p-1}}I_r/I_{r+}.$$
(Here, $\eps$ is merely viewed as a formal variable, making $\F_p[\eps]$ just a free polynomial algebra in one variable.)

Two recent applications of Lahiri and Sorensen's Theorem \ref{LS-split} have relied deeply on explicit descriptions of $\gr I$ when $I$ is a pro-$p$ Iwahori subgroup, notably those in \cite{dottolehung} and \cite{koziol-bharatha-bryan}.  Both of these papers rely on intense ad-hoc computations to produce expressions for these Lie algebras.  (A similar computation of $L_\omega(I)_{\frac1{p-1}+}$ for $I$ the pro-$p$ Iwahori of $\GL_2(K)$ was performed in \cite[Subsection 5.3]{BHH+}.)  One point which has become clear from these applications is that it is of great importance that we better understand these Lie algebras in order to apply Lazard theory to the mod $p$ representation theory of $p$-adic reductive groups, and applications of our Theorem \ref{intro-existence} will likely require this information.  By working more conceptually, we are able to avoid the heavy computations present in \cite{dottolehung} and \cite{koziol-bharatha-bryan}.  

Given any $x\in \mc B(G)$, there is an associated filtered Lie $\O_K$-algebra $\mf g_{x,0}$ defined using integral models of $G$ (see \cite[Section 13.3]{KP}).  We then have the following description of the Lie $\Z_p$-algebra of $G_{x,\frac1{p-1}+}$, proved in Theorem \ref{general-LieAlg-iso} below:
\begin{thm}
There is a natural isomorphism of filtered Lie $\Z_p$-algebras
$$L_\omega(G_{x,\frac1{p-1}+})_{\frac1{p-1}+}\iso \mf g_{x,\frac1{p-1}+}.$$
\end{thm}

To compute the Lie $\F_p[\eps]$-algebra, we reinterpret the Moy-Prasad isomorphism as an isomorphism of Lie $\F_p[\eps]$-algebras.  The kernel of this idea is present in the proof of \cite[Lemma 3.7]{dottolehung}, which still relies on a manual computation (that becomes completely untenable once one is no longer working in the split case). Corollary \ref{MP-cor} below provides the following, where these Moy-Prasad isomorphisms exist by \cite[Theorem 13.5.1]{KP}:
\begin{thm}
The Moy-Prasad isomorphisms
$$\mathrm{MP}_{x,r,r+}: G_{x,r}/G_{x,r+}\xto{\sim} \mf g_{x,r}/\mf g_{x,r+}$$
for $r>\frac1{p-1}$ fit together to form an isomorphism of Lie $\F_p[\eps]$-algebras
\begin{equation}\mathrm{MP}: \gr G_{x,\frac1{p-1}+}\xto{\sim} \gr \mf g_{x,\frac1{p-1}+}.\end{equation}
\end{thm}
Here, we are viewing the Lie $\mc O_F$-algebra $\mf g_{x,\frac1{p-1}+}$ as first being a Lie algebra over $\Z_p$, and then using the natural isomorphism $\gr \Z_p\cong \F_p[\eps]$ sending $p\mapsto \eps$ to view $\gr \mf g_{x,\frac1{p-1}+}$ as a Lie algebra over $\F_p[\eps]$.

\subsection{Acknowledgments}
I am very grateful to Claus Sorensen for his extensive guidance and support while writing this article, and for his many helpful comments.  I also deeply appreciate him initially suggesting the problem which turned into this project.

\subsection{Statement of AI use}
No AI tools were used in the creation of this article; all mathematical results and arguments are the work of the author.

\section{Background}

\subsection{\texorpdfstring{$p$}{p}-valuations} \label{pval-background}

In this subsection, we will review some of the basic ideas of Lazard theory and present a few results which will get used repeatedly throughout the paper.  Although the original source for much of the background material is \cite{lazard65}, we find the reference \cite[Part B]{schneiderbook} much more convenient, so we primarily follow its notation and conventions.


For the convenience of the reader, we recall the definition of a $p$-valuation from \cite[p. 169]{schneiderbook}.

\begin{defn}
Let $G$ be any group.  A \textbf{$p$-valuation} on $G$ is a real-valued function
$$\omega:G\setminus\{1\}\to (0,\infty),$$
along with the convention that $\omega(1)=\infty$, satisfying the following four axioms for any $g,h\in G$:
\begin{enumerate}[{(a)}]
\item $\omega(g)>\frac1{p-1}$,
\item $\omega(g\inv h)\geq \min\{\omega(g),\omega(h)\}$,
\item $\omega\left([g,h]\right)\geq \omega(g)+\omega(h)$,
\item $\omega(g^p)=\omega(g)+1$.
\end{enumerate}
\end{defn}


The following easy fact is used in \cite[III.3.2.5]{lazard65}, but we again provide a proof for the convenience of the reader.

\begin{prop} \label{pth}
Let $G$ be a group admitting a $p$-valuation $\omega$.  Then the $p$th power map $G\to G$, $g\mapsto g^p$ is injective.
\end{prop}

\begin{proof}
By \cite[Proposition 25.1]{schneiderbook}, we know for any $g,h\in G$ that $\omega(g^{-p}h^p)=\omega(g\inv h)+1$.  If $g,h\in G$ satisfy $g^p=h^p$, then $\omega(g\inv h)=-1+\omega(g^{-p}h^p)=\infty$, which implies that $g=h$.
\end{proof}


For the rest of this subsection, all groups admitting a $p$-valuation will be assumed to be compact and the $p$-valuation will be assumed to be topology-defining, which by \cite[p. 181]{schneiderbook} implies that $G$ is pro-$p$.  We also assume that $G$ is of finite rank.

\begin{defn} 
A sequence of elements $(g_1,\ldots, g_r)$ in $G$ is called an \textbf{ordered basis} of $(G,\omega)$ if the continuous function
\begin{align*}
\Z_p^r&\to G\\
(x_1,\ldots, x_r)&\mapsto g_1^{x_1}\cdots g_r^{x_r}
\end{align*}
is a bijection (and hence a homeomorphism by the compactness of $\Z_p^r$), and satisfies the condition for all $x_i\in \Z_p$ that
$$\omega(g_1^{x_1}\cdots g_r^{x_r})=\min_i\{\omega(g_i)+v(x_i)\}$$
where $v$ is the valuation on $\Z_p$ normalized so that $v(p)=1$.
\end{defn}

By \cite[Proposition 26.6]{schneiderbook}, an ordered basis for $(G,\omega)$ always exists, and is of length $\rank(G)$.

\begin{defn} 
$(G,\omega)$ is called \textbf{saturated} if any $g\in G$ such that $\omega(g)>\frac p{p-1}$ is a $p$th power in $G$.
\end{defn}

By \cite[Proposition 26.11]{schneiderbook}, if $(g_1,\ldots, g_r)$ is an ordered basis for $(G,\omega)$, then $(G,\omega)$ is saturated if and only if 
$$\frac1{p-1}<\omega(g_i)\leq \frac p{p-1}\qquad\text{for all }1\leq i\leq r.$$



\begin{rem} \label{OKbasis}
(This is similar to \cite[Example 3.4]{LS}).  Let $K/\Q_p$ be a finite extension equipped with the usual valuation $v$ normalized by $v(p)=1$, and fix any $a>\frac1{p-1}$.  Then it follows immediately from the definition that there is a $p$-valuation $\omega$ on $\O_K$ defined by $\omega(x)\coloneqq v(x)+a$.  Moreover, $(\O_K,\omega)$ is saturated if and only if
$$a\leq \frac1{e_K}+\frac1{p-1},$$
where $e_K=e(K/\Q_p)$ is the absolute ramification degree of $K$.  Indeed, if $x\in \O_K$ satisfies $\omega(x)> \frac{p}{p-1}$, then
$$v(x)=\omega(x)-a> \frac{p}{p-1}-a\geq \frac{p}{p-1}-\frac1{p-1}-\frac1{e_K}=1-\frac1{e_K}.$$
Because $v$ takes values in $\frac1{e_K}\Z$, if $v(x)>1-\frac1{e_K}$, it follows that $v(x)\geq 1$.  It is then immediate that there exists some $y\in \O_K$ such that $x=py$, so indeed $(\O_K,\omega)$ is saturated.  Conversely, if $a> \frac1{e_K}+\frac1{p-1}$, then choosing $x=\varpi_K^{e_K-1}$ for $\varpi_K\in \O_K$ a uniformizer gives us an element of $\O_K$ which is not a multiple of $p$, but which satisfies $\omega(x)> \frac{p}{p-1}$.


We note that $(b_1,\ldots, b_\ell)$ is an ordered basis of $(\O_K,\omega)$ (for any choice of $a>\frac1{p-1}$, so we need not specify in the future) if and only if $(b_i)_{i\leq \ell}$ is a $\Z_p$-basis of $\O_K$ satisfying the equality
$$v\left(\sum_{i=1}^\ell x_ib_i\right)=\min_{i\leq \ell} \{v(b_i)+v(x_i)\}$$
for all $x_i\in \Z_p$.  In such a scenario, we will simply call $(b_i)_{i\leq \ell}$ an \textbf{ordered basis for} $\O_K$.
\end{rem}

\begin{lem} \label{traceless}
Say $p> 2$.  Let $K/\Q_p$ be a finite extension, and let $E/K$ be a quadratic extension with unique nontrivial Galois automorphism denoted by $x\mapsto \ol x$.  Recall that for any $x\in E$, its \textbf{trace} is defined by $x+\ol x$.  Fix an ordered basis $(b_1,\ldots, b_\ell)$ for $\O_K$, and let $\gamma\in \O_E$ be such that $\{1,\gamma\}$ forms an $\O_K$-basis of $\O_E$.  If $\gamma$ is \textbf{traceless}, i.e., if $\gamma+\ol \gamma=0$, then 
$$v(x+y\gamma)=\min\{v(x),v(y\gamma)\}$$
for all $x,y\in K$.  Moreover, $(b_1,\ldots, b_\ell,b_1\gamma, \ldots, b_\ell\gamma)$ is an ordered basis of $\O_E$.
\end{lem}

\begin{proof}
We first claim for any $x,y\in K$ that if $\alpha\coloneqq x+y\gamma$,
$$v(\alpha)=\min\{v(x),v(y)+v(\gamma)\}.$$
Recall that it is always the case that $v(\ol \alpha)=v(\alpha)$, so 
\begin{equation*}
v(\alpha)\leq v(\alpha+\ol\alpha)=v(x+\ol x+y\gamma+\ol{y\gamma})=v(x+y\gamma+x-y\gamma)=v(x),
\end{equation*}
where the last equality follows from $p> 2$.  Similarly, 
$$v(\alpha)\leq v(\alpha-\ol\alpha)=v(y\gamma),$$
hence $v(x+y\gamma)\leq \min\{v(x), v(y\gamma)\}$.  But the opposite inequality is merely the strong triangle inequality, so indeed $v(x+y\gamma)=\min\{v(x),v(y\gamma)\}$ for any $x,y\in \O_K$.  The rest of the result follows immediately from $(b_1,\ldots, b_\ell)$ being an ordered basis of $\O_K$, and $\{1,\gamma\}$ forming an $\O_K$-basis of $\O_E$.
\end{proof}

\begin{rem} \label{tracelessexistence}
We remark that so long as $p> 2$, it is always possible to produce a traceless $\gamma\in \O_E$ such that $\{1,\gamma\}$ forms an $\O_K$-basis of $\O_E$.  Indeed, first choose any $\gamma_0\in \O_E$ such that $\{1,\gamma_0\}$ is an $\O_K$-basis.  Let $\tr(\gamma_0)$ denote the trace of $\gamma_0$, which lies in $\O_K$.  If we set $\gamma\coloneqq \gamma_0-\frac12\tr(\gamma_0)\in \O_E$, then it follows that $\{1,\gamma\}$ is still an $\O_K$-basis of $\O_E$, and by construction $\gamma$ is traceless.
\end{rem}

The following proposition and the subsequent remark will be crucial observations for our descent arguments, and will get used repeatedly throughout the paper.

\begin{prop} \label{invariants}
Let $(G,\omega)$ be a saturated pro-$p$ group.  Let $\theta:G\to G$ be any group homomorphism.  If we define $G^\theta\coloneqq \{g\in G\st\theta(g)=g\}$, which is always a subgroup of $G$, then $(G^\theta,\omega|_{G^\theta})$ is saturated.
\end{prop}

\begin{proof}
We first note that because $G$ is pro-$p$ and of finite rank, the homomorphism $\theta:G\to G$ is automatically continuous by \cite[Corollary 1.21]{DDMS}.  Thus, $G^\theta$ is a closed subgroup of $G$, so is also of finite rank.

Suppose now that $h\in G^\theta$ is such that $\omega(h)>\frac p{p-1}$.  Since $(G,\omega)$ is saturated, this implies that there is some $g\in G$ such that $g^p=h$.  We then find that 
$$\theta(g)^p=\theta(g^p)=\theta(h)=h=g^p.$$
Because taking $p$th powers in a $p$-valuable group is injective by Proposition \ref{pth}, this implies that $\theta(g)=g$.  Therefore, $g\in G^\theta$ is such that $g^p=h$, so we conclude that $G^\theta$ is saturated.
\end{proof}

\begin{rem} \label{gen_invariants}
We briefly note that the above proof immediately generalizes to showing that if $\Theta\sub \End(G)$ is any set of group homomorphisms $G\to G$, then $\omega|_{G^\Theta}$ is a saturated $p$-valuation on 
$$G^\Theta=\{g\in G\st\theta(g)=g \text{ for all }\theta\in \Theta\}.$$
\end{rem}

\begin{defn} \label{grG}
For any real number $v>0$, we write
$$G_v\coloneqq \{g\in G\st\omega(g)\geq v\}\quad \mathrm{and}\quad G_{v+}\coloneqq \{g\in G\st\omega(g)>v\}.$$
These are normal subgroups of $G$, and form a decreasing exhaustive separated filtration of $G$ satisfying $[G_v,G_{v'}]\sub G_{v+v'}$ for all $v,v'>0$.  For each $v>0$, define $\gr_v G\coloneqq G_v/G_{v+}$, which is an abelian group.  We then define the associated graded abelian group
$$\gr G\coloneqq \bigoplus_{v>0} \gr_v G.$$
As in \cite{dottolehung}, we write $\symb(g)\in \gr G$ for the homogeneous element in $\gr_{\omega(g)}G$ which is the image of $g$.
\end{defn}

\begin{rem}
As explained in \cite[Section 25]{schneiderbook}, $\gr G$ is naturally an $\F_p[\eps]$-Lie algebra, where 
$$[\symb(g),\symb(h)]=[g,h]G_{\omega(g)+\omega(h)}$$
for $g,h\in G$, and $\eps(\symb(g))=g^p G_{\omega(g)+1}=\symb(g^p)$.  Moreover, this is a free $\F_p[\eps]$-module: given an ordered basis $(g_1,\ldots, g_r)$ of $(G,\omega)$, we have an $\F_p[\eps]$-basis of $\gr G$ given by $\left(\symb(g_i)\right)_{i\leq r}$.
\end{rem}

Assume that $G$ is saturated.  As described in \cite[Section 31]{schneiderbook}, there also exists a filtered Lie $\Q_p$-algebra $L_\omega(G)$ of dimension equal to $\rank(G)$.  Within it, there is a $\Z_p$-lattice $L_\omega(G)_{\frac1{p-1}+}$, itself a Lie $\Z_p$-algebra under the same Lie bracket, which comes with mutually inverse $\log$ and $\exp$ maps
\[\begin{tikzcd}
	{L_\omega(G)_{\frac1{p-1}+}} & G
	\arrow["\exp", shift left=2, from=1-1, to=1-2]
	\arrow["\log", shift left=2, from=1-2, to=1-1]
\end{tikzcd}\]
that preserve the respective filtrations.  That is, if we let $\tilde\omega$ denote the filtration on the Lie algebra, then $\tilde\omega(\log(g))=\omega(g)$ for all $g\in G$; in particular, for any $r>\frac1{p-1}$, $\log$ and $\exp$ restrict to mutual inverses between $L_\omega(G)_r$ and $G_r$.  Then we have the following standard result (e.g., it is used in \cite[Subsection 5.3]{BHH+}), which we prove for convenience of the reader.

\begin{prop} \label{log-exp}
Assume $G$ is saturated.  The filtration-preserving $\log$ and $\exp$ induce well-defined mutually inverse isomorphisms of Lie $\F_p[\eps]$-algebras
\[\begin{tikzcd}
	{\gr L_\omega(G)_{\frac1{p-1}+}=\bigoplus_{v>\frac1{p-1}} L_\omega(G)_v/L_\omega(G)_{v+}} & \bigoplus_{v>\frac1{p-1}}G_v/G_{v+}=\gr G.
	\arrow["\gr\exp", shift left=2, from=1-1, to=1-2]
	\arrow["\gr\log", shift left=2, from=1-2, to=1-1]
\end{tikzcd}\]
\end{prop}
Here, because the $\Z_p$-action on $L_\omega(G)_{\frac1{p-1}+}$ satisfies $\tilde\omega(x\lambda)=v(x)+\tilde\omega(\lambda)$ for $x\in \Z_p$ and $\lambda\in L_\omega(G)_{\frac1{p-1}+}$, then $\gr L_\omega(G)_{\frac1{p-1}+}$ is naturally a Lie algebra over the ring $\gr \Z_p\iso \F_p[\eps]$, where $\eps$ acts via a degree-one map on this graded space (corresponding to scaling by $p$).


\begin{proof}
We will show that $\gr \log$ is well-defined, $\F_p[\eps]$-linear, and preserves the Lie bracket; one can easily show using similar arguments that the same holds for $\gr \exp$.

By construction, $\gr\log$ is the direct sum of the homogeneous maps 
$$\gr_v\log:G_v/G_{v+}\to L_\omega(G)_v/L_\omega(G)_{v+}$$
defined by taking any representative $g\in G_v$ of a coset, and setting $\gr_v\log(g)$ to be the coset $\log(g)+L_\omega(G)_{v+}$, which a priori need not be well-defined.  However, this follows from the following fact: given any $g,h\in G$, 
$$\tilde\omega\left(\log(gh)-\log(g)-\log(h)\right)>\max\{\omega(g),\omega(h)\}.$$
Indeed, if $g\in G_v$ and $h\in G_{v+}$, then $\log(h)\in L_\omega(G)_{v+}$, so this fact implies that $\log(gh)-\log(g)\in L_\omega(G)_{v+}$.  This fact similarly immediately proves that each $\gr_v\log$ is $\F_p$-linear.

To prove this fact, we note that an easy inductive argument using both parts of \cite[Proposition 25.1]{schneiderbook} shows for any $n\geq 1$ that
\begin{equation}\label{25.1gen}
\omega(h^{-p^n}g^{-p^n}(gh)^{p^n})>\max\{\omega(g),\omega(h)\}+n.\end{equation}
We also have by \cite[III.2.1.10]{lazard65} that one can compute $\log(x)+\log(y)$ for $x,y\in G$ by letting $a_n\in G$ for $n\geq 1$ be the unique elements satisfying $a_n^{p^n}=x^{p^n}y^{p^n}$ (which exist and are unique by the saturation of $G$), noting that $(a_n)_n$ is a convergent sequence in $G$ say with limit $a\in G$, in which case $\log(a)$ is the desired value.  To ease notation, one can write this as
$$\log(x)+\log(y)=\log\left(\lim_{n\to \infty}\left(x^{p^n}y^{p^n}\right)^{1/p^n}\right).$$

It follows that
\begin{align*}
\log(gh)-\log(g)-\log(h)&=\log(h\inv)+\log(g\inv)+\log(gh)
\end{align*}
has by \ref{25.1gen} valuation equal to
\begin{align*}
\lim_{n\to \infty} \omega\left(\left(h^{-p^n}g^{-p^n}(gh)^{p^n}\right)^{1/p^n}\right)
&=\lim_{n\to \infty}\omega(h^{-p^n}g^{-p^n}(gh)^{p^n})-n\\
&>\max\{\omega(g),\omega(h)\}.
\end{align*}
(The inequality remains strict despite the limit due to the values attained by $\omega$ being a discrete set.)  We have thus shown that $\gr\log$ is well-defined and $\F_p$-linear; that it is $\F_p[\eps]$-linear follows from \cite[Lemma 31.1]{schneiderbook}, as we know that $\log(g^x)=x\log(g)$ for any $g\in G$, $x\in \Z_p$.

What remains is to prove that $\gr\log$ preserves the Lie bracket.  Using \cite[Exercice III.2.1.10]{lazard65} and the same shorthand as before, we have for any $g,h\in G$
$$[\log(g),\log(h)]\coloneqq \log(g)\log(h)-\log(h)\log(g)=\log\left(\lim_{n\to \infty}[g^{p^n},h^{p^n}]^{1/p^{2n}}\right).$$
Thus, one can compute that
$$\tilde\omega\left(\log([g,h])-[\log(g),\log(h)]\right)=\lim_{n\to \infty} \omega\left([g,h]^{p^{2n}}[g^{p^n},h^{p^n}]\inv\right)-2n,$$
which for any value of $n$ is strictly greater than $\omega(g)+\omega(h)$ by a repeated application of \cite[Proposition 25.3]{schneiderbook}.  As usual, this strict inequality also holds in the limit by the discreteness of $\omega$, proving immediately that the Lie bracket is preserved by $\gr\log$.
\end{proof}


We will use the following technical observation in later arguments.

\begin{lem} \label{limits}
Let $(\omega_n)_{n\geq 1}$ be a sequence of $p$-valuations on a pro-$p$ group $G$.  Suppose that $(\omega_n(g))_{n\geq 1}$ is a convergent sequence, bounded away from $\frac1{p-1}$, for each $g\in G$.  Then $\omega(g)\coloneqq\lim_{n\to \infty} \omega_n(g)$ defines a $p$-valuation on $G$.
\end{lem}

\begin{proof}
Since every $(\omega_n(g))_n$ is bounded away from $\frac1{p-1}$ and convergent, then each $\omega(g)$ lies in $\left(\frac1{p-1},\infty\right)$.  The remaining pieces of the definition of being a $p$-valuation follow immediately from taking limits.
\end{proof}


\subsection{Review of quasisplit reductive groups} \label{qsplit-review}

We will now review the structure of quasisplit reductive groups over $p$-adic fields, with an emphasis on their root groups.  A great reference for almost everything in this subsection is \cite[Section 2.6]{KP}.

\begin{notn}
Let $p>2$.  Let $G$ be any connected reductive quasisplit group defined over a fixed finite extension $K/\Q_p$.  Pick a maximally $K$-split maximal $K$-torus $T\leq G$, with $S\leq T$ its maximal $K$-split subtorus; since $G$ is quasisplit, $T$ is the centralizer of $S$ in $G$.  Let $(X^*(S),\Phi, X_*(S), \Phi^\vee)$ be its associated relative root datum.  We also let $(X^*(T),\tilde\Phi, X_*(T),\tilde\Phi^\vee)$ denote the absolute root datum.  We also fix compatible systems of positive and simple relative/absolute roots: $\Phi^+\sub \Phi$, $\tilde\Phi^+\sub \tilde\Phi$, $\Delta\sub \Phi^+$, and $\tilde\Delta\sub \tilde\Phi^+$.  In particular, we have a $\Theta\coloneqq \Gal(\ol K/K)$-action on $\tilde\Phi$, and we can choose everything so that $\tilde\Delta$ is $\Theta$-invariant, and hence corresponds to a Borel $K$-subgroup $B\leq G$ containing $T$.

For each $\tilde\alpha\in \tilde\Phi$, we let $K_{\tilde\alpha}/K$ be the minimal finite extension over which $\tilde\alpha$ is defined; more concretely, take the stabilizer $\Theta_{\tilde\alpha}$ of $\tilde\alpha$, and let $K_{\tilde\alpha}$ denote its fixed field.

Since $G$ is quasisplit, we have a projection map $\tilde\Phi\onto \Phi$ given by restriction to $S$, and the preimage of any element of $\Phi$ is a $\Theta$-orbit of an element in $\tilde\Phi$.  In particular, if $\alpha\in \Phi$, different choices of $\tilde\alpha\in \tilde\Phi$ lying above $\alpha$ have fields of definitions $K_{\tilde\alpha}$ which are Galois conjugates.

For each $\alpha\in \Phi$, there is an associated unipotent root $K$-subgroup $U_\alpha$ normalized by $S$ whose Lie algebra is the direct sum of the weight spaces for the action of $S$ on $\mathrm{Lie}(G)$ which are $\Z^+$-multiples of $\alpha$.  (That is, if $\alpha$ is multipliable, we by convention take $U_{2\alpha}\sub U_\alpha$.). Similarly, for $\tilde\alpha\in \tilde\Phi$, we can define a unipotent $K_{\tilde\alpha}$-subgroup $U_{\tilde\alpha}$ normalized by $T$ in the same way.

Given our tori $S\leq T\leq G$, we induce tori $S\ad\leq T\ad\leq G\ad$, $S\der\leq T\der\leq G\der$, and $S\sc\leq T\sc\leq G\sc$ which are maximal $K$-split tori and maximally $K$-split maximal tori respectively, given by taking the image of our tori under $G\to G\ad$, intersecting our tori with $G\der\leq G$, and taking the preimage of our tori under the composition $G\sc\to G\der \to G$.

By \cite[Section 2.6(a)]{KP}, the absolute (resp., relative) root systems of $G$ are naturally identified with those of $G\ad$, $G\der$, and $G\sc$ via the natural maps
$$G\sc\to G\der\to G\to G\ad.$$
Thus, we will not distinguish between any of these in our notation.
\end{notn}

\begin{rem}
(See \cite[pp. 96-97]{KP}.)  Let $\alpha\in \Phi$.  We say $\alpha$ is \textbf{multipliable} if $2\alpha\in \Phi$, and \textbf{divisible} if there is some $\beta\in \Phi$ such that $\alpha=2\beta$.  We will classify the elements of $\Phi$ into three types as follows:
\begin{itemize}
\item $\alpha$ is of type R1 if it is neither multipliable nor divisible.
\item $\alpha$ is of type R2 if it is multipliable.
\item $\alpha$ is of type R3 if it is divisible.
\end{itemize}
We write $\Phi^{\mathrm{nd}}$ for the set of nondivisible roots, i.e., those of type R1 or R2.  This classification allows us to describe $U_\alpha(K)$ rather explicitly.
\end{rem}

We fix once and for all a Chevalley-Steinberg system $X$ for $G\ad$, i.e., a Chevalley system $X=\{X_{\tilde\alpha}\in \mf g_{\tilde\alpha}(\ol K) \st \tilde\alpha\in \tilde\Phi\}$ which satisfies the following invariance property for any $\tilde\alpha\in \tilde\Phi$ restricting to any $\alpha\in \Phi$, and for any $\sigma\in \Theta$:
\begin{enumerate}[{(i)}]
\item If $\alpha$ is of type R1 or R2, then 
$$\sigma(X_{\tilde\alpha})=X_{\sigma\tilde\alpha}.$$
\item If $\alpha$ is of type R2 and $\tilde\alpha'\in \Theta\cdot \tilde\alpha$ is the unique root such that $\tilde\beta=\tilde\alpha+\tilde\alpha'\in \tilde\Phi$, and $\sigma\in \Theta_{\tilde\beta}$, then 
$$\sigma(X_{\tilde\beta})=\eps(\sigma)\cdot X_{\tilde\beta},$$
where $\eps:\Theta_{\tilde\beta}/\Theta_{\tilde\alpha}\to \{\pm 1\}$ is the unique nontrivial character.
\end{enumerate}
This exists by \cite[Lemma 2.9.14]{KP}.

Given our Chevalley-Steinberg system $X$, we get for each $\tilde\alpha\in \tilde\Phi$ an isomorphism of $K_{\tilde\alpha}$-groups
$$\G_a\xto\sim U_{\tilde\alpha}$$
which is uniquely determined by requiring its differential to send $1$ to $X_{\tilde\alpha}$.  By \cite[Construction 2.9.16]{KP}, the $\Theta$-invariance properties of our Chevalley-Steinberg system allow it to induce isomorphisms as follows.

If $\alpha\in \Phi$ is of type R1, then the preimage of $\alpha$ in $\tilde\Phi$ consists of a set $\{\tilde\alpha_1,\ldots, \tilde\alpha_{n_\alpha}\}$ of pairwise orthogonal elements, and for any $\tilde\alpha\in \tilde\Phi$ lying above $\alpha$, our Chevalley system $X$ induces an isomorphism $R_{K_{\tilde\alpha}/K}U_{\tilde\alpha}\xto{\sim} U_\alpha$ of $K$-groups.  In particular, we get an isomorphism
$$u_\alpha:K_{\tilde\alpha}\xto{\sim} U_\alpha(K).$$

Note that $\alpha\in \Phi$ is of type R2 if and only if $2\alpha\in \Phi$ is of type R3.  In such a case, the preimage of $\alpha$ in $\tilde\Phi$ is 
$$\{\tilde\alpha_1,\tilde\alpha_1',\ldots, \tilde \alpha_{n_\alpha}, \tilde\alpha_{n_\alpha}'\}$$
with $\tilde\alpha_i$ orthogonal to $\tilde\alpha_j$ and $\tilde\alpha_j'$ if $i\neq j$.  If we set $\tilde\beta_i\coloneqq\tilde\alpha_i+\tilde\alpha'_i$ for all $i$, then the preimage of $2\alpha$ in $\tilde\Phi$ is precisely
$$\{\tilde\beta_1,\ldots, \tilde\beta_{n_\alpha}\}.$$
For any such $\tilde\alpha, \tilde\alpha', \tilde\beta=\tilde\alpha+\tilde\alpha'$, $K_{\tilde\alpha}/K_{\tilde\beta}$ is now a quadratic Galois extension, and our Chevalley-Steinberg system $X$ induces isomorphisms of $K$-groups
$$\mathrm{R}_{K_{\tilde\beta}/K}U_{K_{\tilde\alpha}/K_{\tilde\beta}}\xto{\sim}U_\alpha,\qquad R_{K_{\tilde\beta}/K} U^0_{K_{\tilde\alpha}/K_{\tilde\beta}}\xto{\sim} U_{2\alpha}$$
where $U_{K_{\tilde\alpha}/K_{\tilde\beta}}$ and $U_{K_{\tilde\alpha}/K_{\tilde\beta}}^0$ have $R$-points for a $K_{\tilde\beta}$-algebra $R$ given by
$$U_{K_{\tilde\alpha}/K_{\tilde\beta}}(R)=\left\{(x,y)\in (R\otimes_{K_{\tilde\beta}}K_{\tilde\alpha})^2\Bigst x\ol x=y+\ol y\right\}$$
and
$$U_{K_{\tilde\alpha}/K_{\tilde\beta}}^0(R)=\{y\in R\otimes_{K_{\tilde\beta}} K_{\tilde\alpha}\st y+\ol y=0\}.$$
(Here, $\ol{(\cdot)}\in \Gal(K_{\tilde\alpha}/K_{\tilde\beta})$ is the unique nontrivial Galois automorphism extended to $R\otimes_{K_{\tilde\beta}}K_{\tilde\alpha}$.)  The group structure on $U_{K_{\tilde\alpha}/K_{\tilde\beta}}(R)$ is given by
$$(x,y)\cdot (x',y')=(x+x', y+x\ol{x'}+y').$$
We thus have from our Chevalley-Steinberg system an isomorphism
$$u_\alpha:U_{K_{\tilde\alpha}/K_{\tilde\beta}}(K_{\tilde\beta})=\{(x,y)\in K_{\tilde\alpha}^2\st x\ol x=y+\ol y\}\xto{\sim}U_\alpha(K).$$
Note that these isomorphisms satisfy the property that
$$su_\alpha(x)s\inv=u_\alpha(\alpha(s)x)$$
if $s\in S(K_{\tilde\alpha})$ and $x\in K_{\tilde\alpha}$ when $\alpha$ is of type R1, and
$$su_\alpha(x,y)s\inv=u_\alpha(\alpha(s)x, (2\alpha)(s)y)$$
if $s\in S(K_{\tilde\alpha})$ and $(x,y)\in U_{K_{\tilde\alpha}/K_{\tilde\beta}}(K_{\tilde\beta})$ when $\alpha$ is of type R2.

\begin{rem} \label{Chev-Stein-isoms}
(See \cite[p. 97]{KP}.)  Our Chevalley-Steinberg system gives similar isomorphisms $u_{\tilde\alpha}:K_{\tilde\alpha}\xto{\sim}U_{\tilde\alpha}(K_{\tilde\alpha})$ from the $K_{\tilde\alpha}$-group isomorphisms $\G_a\xto{\sim} U_{\tilde\alpha}$.  We have the following concrete descriptions of how these two collections of isomorphisms interact.
\begin{itemize}
\item If $\alpha$ is of type R1, fix $\tilde\alpha\in \tilde\Phi$ lying above $\alpha$, so now the finite set $\{\sigma\tilde\alpha\st \sigma\in \Theta/\Theta_{\tilde\alpha}\}$ is exactly the roots lying above $\alpha$.  Then for $x\in K_{\tilde\alpha}$,
$$u_\alpha(x)=\prod_{\sigma\in \Theta/\Theta_{\tilde\alpha}} u_{\sigma\tilde\alpha}(\sigma(x)).$$
Moreover, these terms commute (in fact, $U_\alpha$ is commutative).
\item If $\alpha$ is of type R2, or equivalently $2\alpha\in \Phi$ is of type R3, we fix $\tilde\alpha$ lying above $\alpha$, which determines a unique other root $\tilde\alpha'$ also lying above $\alpha$ which is not orthogonal to $\tilde\alpha$.  Now $\tilde\beta\coloneqq \tilde\alpha+\tilde\alpha'\in \tilde\Phi$ lies above $2\alpha$.  The preimages of $\alpha$ and $2\alpha$ in $\tilde\Phi$ are respectively
$$\{\sigma\tilde\alpha, \sigma\tilde\alpha'\st \sigma\in \Theta/\Theta_{\tilde\beta}\},\qquad\{\sigma\tilde\beta\st \sigma\in \Theta/\Theta_{\tilde\beta}\}.$$
As usual, let $\ol{(\cdot)}$ denote the nontrivial automorphism in $\Gal(K_{\tilde\alpha}/K_{\tilde\beta})$, which switches $\tilde\alpha$ and $\tilde\alpha'$.  Then for $(x,y)\in U_{K_{\tilde\alpha}/K_{\tilde\beta}}(K_{\tilde\beta})$,
\begin{align*}
u_\alpha(x,y)&=\prod_{\sigma\in \Theta/\Theta_{\tilde\beta}} \sigma\left(u_{\tilde\alpha}(x)u_{\tilde\alpha'}(\ol x)u_{\tilde\beta}(y)\right)\\
&=\prod_{\sigma\in \Theta/\Theta_{\tilde\beta}} u_{\sigma\tilde\alpha}(\sigma(x))u_{\sigma\tilde\alpha'}(\sigma(\ol x))u_{\sigma\tilde\beta}(\sigma(y)).
\end{align*}
Here, these three-term pieces indexed by $\sigma\in \Theta/\Theta_{\tilde\beta}$ commute with one another.  (We also note that we are abusing notation slightly in the final line, as strictly speaking $\sigma(x)$ for $x\in K_{\tilde\alpha}$ is not well-defined for $\sigma\in \Theta/\Theta_{\tilde\beta}$, however the three-term piece as a whole is fixed under the action of $\Theta_{\tilde\beta}$, so one could merely fix a set of representatives of $\Theta/\Theta_{\tilde\beta}$ and the result is independent of this choice.)
\end{itemize}
\end{rem}


\begin{rem} \label{root-filts}
Our Chevalley-Steinberg system $X$ induces in the usual way (e.g., see \cite[Construction 6.1.17]{KP}) a Chevalley valuation of the root datum, which is an absolutely special point $x_0$ of the semisimple apartment $\mc A(S)$ of $G$ associated to $S$ by \cite[Proposition 7.11.4]{KP}.  We can translate $x_0$ by any element of the vector space $V(S)\coloneqq \R\otimes_\Z X_*(S)$; this action factors through $V(A_G)$ for $A_G\leq G$ the maximal $K$-split torus in the center of $G$ (see \cite[p. 200]{KP}), giving in fact an action of $V(S/A_G)$.  Now $S/A_G$ is isogenous to a maximal $K$-split torus $S\ad$ in $T\ad$, the image of $T$ in $G\ad$, which is a maximally split maximal $K$-torus.  Thus, 
$$V(S)/V(A_G)=V(S/A_G)=V(S\ad)=V(T\ad)^\Theta.$$
In fact, $\mc A(S)$ is an affine space over these equal vector spaces.  Given any $x\in \mc A(S)$, we have an associated filtration on the root groups given as follows.  (Note that $U_{\alpha,x,r}$ is used in two different ways in \cite{KP}; our notation agrees with \cite[7.3.3]{KP} with the constant function $f=r$, but disagrees with \cite[6.1.2]{KP}.)  Note that there is a unique $v\in V(S\ad)$ such that $x=x_0+v$.

If $\alpha\in \Phi$ is of type R1, then for $r\geq 0$,
$$U_{\alpha,x,r}=\{u_\alpha(x)\st x\in K_{\tilde\alpha}, v(x)+\langle \alpha,v\rangle\geq r\}.$$
We also write 
$$U_{\alpha,x,r+}=\bigcup_{s>r} U_{\alpha, x, s}=\{u_\alpha(x)\st x\in K_{\tilde\alpha}, v(x)+\langle \alpha,v\rangle> r\}.$$
If $\alpha\in \Phi$ is of type R2, then we have for $r\geq 0$ the filtration
$$U_{\alpha,x,r}= \Big\{u_\alpha(x,y)\Bigst x,y\in K_{\tilde\alpha}, x\ol x=y+\ol y, \quad \min\{v(x)+\langle \alpha,v\rangle, v(y)+2\langle \alpha,v\rangle\}\geq r\Big\}.$$
We also have
\begin{align*}
U_{\alpha,x,r+}&= \bigcup_{s>r} U_{\alpha,x,s}\\
&=\Big\{u_\alpha(x,y)\Bigst x,y\in K_{\tilde\alpha}, x\ol x=y+\ol y, \quad \min\{v(x)+\langle \alpha,v\rangle, v(y)+2\langle \alpha,v\rangle\}> r\Big\}.
\end{align*}
\end{rem}

In all of the above, we denote the usual additive valuation on $\ol{\Q_p}$ by $v$, normalized so that $v(p)=1$ (there should be no situation where confusion could arise between elements $v\in V(S\ad)$ and the valuation $v$).

\section{\texorpdfstring{$p$}{p}-valuations on root groups}

In this section, we will study the root groups of a quasisplit $p$-adic reductive group.  More specifically, we take the exact same setup as in Subsection \ref{qsplit-review}, and what we seek to prove is the following for any $\alpha\in \Phi^{\mathrm{nd}}$; the proof of this result will take the remainder of this section.

\begin{prop} \label{root-group-pval}
Let $x\in \mc A(S)$ be arbitrary, and assume $p> 2$.  For $u\in U_{\alpha,x,\frac1{p-1}+}$, define
$$\omega(u)\coloneqq \min\left\{r>\frac1{p-1} \Biggst u\in U_{\alpha,x,r}\right\}.$$
Then $\omega$ is a $p$-valuation on $U_{\alpha,x,\frac1{p-1}+}$, and moreover it is saturated.
\end{prop}

We note that because the jump set of the filtration on $U_\alpha(K)$ is discrete (see e.g. \cite[Fact 6.4.5]{KP}), this minimum is attained, justifying the notation.  

To prove this, recall that $x=x_0+v$ for $v\in V(S\ad)$, where $x_0$ arises from our fixed Chevalley-Steinberg system.

\subsection{Root groups of type R1} \label{R1}

We first prove a useful lemma.

\begin{lem} \label{OKsaturation}
Let $F/\Q_p$ be any finite extension.  Let $a\in \R$ be arbitrary, and let $i$ be the least integer strictly greater than
$$e_F\cdot \left(\frac1{p-1}-a\right),$$
so that $\mf m_F^i=\{x\in F\st v(x)>\frac1{p-1}-a\}$.  Then the function $\omega$ on $\mf m_F^i$ defined by
$$\omega(x)\coloneqq v(x)+a$$
is a saturated $p$-valuation.
\end{lem}

\begin{proof}
Recall that our valuation is normalized so that $v(p)=1$, so the only nontrivial part of verifying that this is a $p$-valuation is checking that $\omega(x)>\frac1{p-1}$; this follows from the choice of $i$, as
$$\omega(x)=v(x)+a\geq \frac i{e_F}+a>\frac1{p-1}$$
for any $x\in \mf m_F^i$.  

To see that it is a saturated $p$-valuation, suppose that $x\in \mf m_F^i$ is not a $p$th power (i.e., not a multiple of $p$), which directly implies that $x\notin \mf m_F^{i+e_F}$, or in other words that
$$v(x)<\frac{i}{e_F}+1.$$
Thus, as $v(x)\in \frac1{e_F}\Z$, we know that $v(x)\leq \frac{i-1}{e_F}+1$.  It follows then from our choice of $i$ that
$$\omega(x)=v(x)+a\leq \frac1{p-1}+1=\frac{p}{p-1},$$
as otherwise we could have chosen a smaller value of $i$, so $\omega$ is indeed saturated.
\end{proof}

Suppose now that $\alpha\in \Phi^{\mathrm{nd}}$ is of type R1.  In this case, we have an isomorphism
$$u_\alpha: K_{\tilde\alpha}\xto{\iso} U_\alpha(K),$$
and the groups $U_{\alpha, x, r}$, $U_{\alpha,x,r+}$ are given by
$$U_{\alpha, x, r}=\{u_\alpha(x)\st v(x)+\langle \alpha, v\rangle\geq r\},\quad U_{\alpha, x, r+}=\{u_\alpha(x)\st v(x)+\langle \alpha, v\rangle> r\}.$$
Thus, we have an induced isomorphism $u_\alpha: \mf m_{K_{\tilde\alpha}}^i\xto{\iso} U_{\alpha, x, \frac1{p-1}+}$ where $i$ is the least integer strictly greater than  
$$e_{K_{\tilde\alpha}}\cdot \left(\frac1{p-1}-\langle \alpha,v\rangle\right).$$
Under this isomorphism, the function $\omega$ defined in Proposition \ref{root-group-pval} pulls back to the function on $\mf m_{K_{\tilde\alpha}}^i$ given by
$$\omega(x)=v(x)+\langle \alpha, v\rangle.$$
It is then immediate from the above Lemma \ref{OKsaturation} that $\omega$ is a saturated $p$-valuation on $U_{\alpha,x,r}$.  This completes the proof of Proposition \ref{root-group-pval} when $\alpha$ is of type R1.

\subsection{Root groups of type R2} \label{R2}

Suppose now that $\alpha$ is of type R2.  To simplify notation, write $a\coloneqq \langle \alpha, v\rangle$, as we will be repeatedly writing this.

For convenience, we will work with the root group of $\SU_3$ corresponding to the unique positive nondivisible root (see \cite[Section 3.2]{KP} for details), although the arguments directly generalize to our more general setup.  As always, we assume $p>2$, so we may freely apply Lemma \ref{traceless}.

To fix notation, let $F/\Q_p$ be a finite field extension, $E/F$ a quadratic field extension, and $v$ the usual valuation on $E$ normalized so that $v(p)=1$.  Let $x\mapsto \ol x$ denote the unique nontrivial Galois automorphism of $E/F$.  For any $x,y\in E$, let $u_\alpha(x,y)\in \SU_3(F)\sub \SL_3(E)$ be defined by
$$u_\alpha(x,y)\coloneqq \begin{pmatrix}
1&x&y\\
0&1&\ol x\\
0&0&1
\end{pmatrix}.$$
One can check that $u_\alpha(x,y)\cdot u_\alpha(x',y')=u_\alpha(x+x',y+x\ol{x}'+y')$.  Then 
$$U_\alpha(F)\coloneqq \{u_\alpha(x,y)\st x,y\in F, x\ol x=y+\ol y\}$$
is the unique positive nondivisible root group of $\SU_3(F)$.  To recover our case of interest, take $E=K_{\tilde\alpha}$ and $F=K_{\tilde\beta}$.

\begin{rem} \label{R2-desc}
Under the isomorphism between $U_\alpha(F)$ and our original root group of type R2 in $G(K)$ (in the case when $F=K_{\tilde\beta}$) arising from our Chevalley-Steinberg system, the subgroup of $U_\alpha(F)$ corresponding to $U_{\alpha,x,\frac1{p-1}+}$ will be denoted in this subsection by $U_{\alpha}^a$, and can be described by
\begin{align*}
U_{\alpha}^a&=\left\{u_\alpha(x,y)\in U_\alpha(F)\Biggst v(x)>\frac1{p-1}-a,\quad v(y)>\frac1{p-1}-2a\right\}\\
&=\left\{u_\alpha(x,y)\in U_\alpha(F)\biggst x\in \mf m_E^i, y\in \mf m_E^j\right\},
\end{align*}
where $i$ is the least integer strictly greater than $e_F\cdot \left(\frac1{p-1}-a\right)$ and $j$ is the least integer strictly greater than $e_F\cdot \left(\frac1{p-1}-2a\right)$.  (We remark that $2i\geq j$, so $U_{\alpha}^a$ is easily seen to indeed be a subgroup.  This description follows from \cite[Definition 7.3.3]{KP} and \cite[Lemma 6.1.6]{KP}.)  We also note for future reference that, in fact, this same description applies to $U_{\alpha,x,r+}$ for any $r\geq 0$; one simply replaces $\frac1{p-1}$ by $r$ wherever it appears.

The function $\omega$ in Proposition \ref{root-group-pval} corresponds to the function on $U_{\alpha}^a$ defined by
$$\omega(u_\alpha(x,y))=\min\{v(x)+a,v(y)+2a\}.$$
We claim that it is a $p$-valuation, and will show later in this subsection that it is saturated.
\end{rem}


\begin{lem}
The function $\omega$ defines a $p$-valuation on $U_{\alpha}^a$.
\end{lem}

\begin{proof}
Suppose first that $a\in \Q$.  By a mild generalization of \cite[Example 23.3]{schneiderbook} using the exact same argument, one can construct a $p$-valuation on the standard pro-$p$ Iwahori subgroup of $\GL_n(\O_E)$ given that $ne<p-1$.  (This is also \cite[III.3.2.7]{lazard65}.)  However, the only requirement for that proof to go through for a subgroup of $\GL_n(E)$ consisting of upper triangular unipotent matrices is that $(j-i)a+v(a_{ij})>\frac1{p-1}$ for any $i<j$, i.e., in our case that $v(x)+a>\frac1{p-1}$ and $v(y)+2a>\frac1{p-1}$ for $u_\alpha(x,y)\in U_{\alpha}^a$.  If this condition is satisfied, the claimed function is immediately shown to be a $p$-valuation.  This property follows directly from our above description of $U_{\alpha}^a$.  One can then use Lemma \ref{limits} to allow $a$ to lie outside of $\Q$: letting $i$ and $j$ be the integers in Remark \ref{R2-desc}, any $a'$ satisfying
$$a\geq a'>\max\left\{\frac1{p-1}-\frac{i}{e_E}, \frac12\left(\frac1{p-1}-\frac{j}{e_E}\right)\right\}$$
will have $U_\alpha^{a'}=U_\alpha^a$.  (The definitions of $i$ and $j$ force this interval to be nonempty.)  Now choose a sequence of such $a'$ which are rational and converging to $a$, and then apply Lemma \ref{limits}.
\end{proof}


The following convenient definition is made in \cite[Definition 3.3]{LS}.

\begin{defn}
Let $H$ be a $p$-valuable group and let $(H_1,\ldots, H_s)$ be a list of subgroups such that the multiplication map
$$m:H_1\times\cdots\times H_s\to H$$
is bijective.  We say a $p$-valuation $\omega$ on $H$ is \textbf{compatible} with $(H_1,\ldots H_s)$ if 
$$\omega(h_1\cdots h_s)=\min\{\omega(h_1),\ldots, \omega(h_s)\}$$
for all $h_i\in H_i$.
\end{defn}

\begin{prop} \label{compatibility}
Using Remark \ref{tracelessexistence}, fix a traceless element $\gamma\in \O_E$ such that $\{1,\gamma\}$ is an $\O_F$-basis of $\O_E$.  Consider the following closed subgroups of $U_\alpha^a$:
\begin{align*}
U_{\alpha,1}^a&\coloneqq\left\{u_\alpha\left(x,\frac12x^2\right)\Biggst x\in F, \quad v(x)>\frac1{p-1}-a\right\},\\ U_{\alpha,2}^a&\coloneqq \left\{u_\alpha\left(y\gamma,-\frac12y^2\gamma^2\right)\Biggst y\in F, \quad v(y)>\frac1{p-1}-a-v(\gamma)\right\},\\
U_{\alpha,3}^a&\coloneqq \left\{u_\alpha(0,z\gamma)\Biggst z\in F, \quad v(z)>\frac1{p-1}-2a-v(\gamma)\right\}.
\end{align*}
The multiplication map $U_{\alpha,1}^a\times U_{\alpha,2}^a\times U_{\alpha,3}^a\to U_\alpha^a$ is bijective, and the $p$-valuation $\omega$ on $U_\alpha^a$ is compatible with $(U_{\alpha,1}^a,U_{\alpha,2}^a,U_{\alpha,3}^a)$.
\end{prop}

\begin{proof}
It is easy to check that these are indeed all closed subgroups of $U_\alpha^a$, so we will omit this.  In fact, each is isomorphic to $\mf m_F^n$ as a topological group for appropriate values of $n$, via the respective maps $x\mapsto u_\alpha\left(x,\frac12x^2\right)$, $y\mapsto u_\alpha\left(y\gamma,-\frac12y^2\gamma^2\right)$, and $z\mapsto u_\alpha(0,z\gamma)$.  Concretely, if we write $n_1,n_2,n_3$ for the least integers strictly greater than 
$$e_F\cdot \left(\frac1{p-1}-a\right),\qquad e_F\cdot \left(\frac1{p-1}-a-v(\gamma)\right),\qquad e_F\cdot\left( \frac1{p-1}-2a-v(\gamma)\right),$$
respectively, then we have isomorphisms $\mf m_F^{n_i}\xto{\sim} U_{\alpha,i}^a$ for each $1\leq i \leq 3$.

First, we will show that the multiplication map is bijective; it is sufficient to show that the composition $\phi:\mf m_F^{n_1}\times \mf m_F^{n_2}\times \mf m_F^{n_3}\to U_{\alpha}^a$ of the multiplication map with the above three isomorphisms is a bijection.  A brief computation reveals that $\phi$ is explicitly given by
\begin{align*}
\phi:\mf m_F^{n_1}\times \mf m_F^{n_2}\times \mf m_F^{n_3}&\to G\\
(x,y,z)&\mapsto u_\alpha\left(x,\frac12x^2\right)\cdot u_\alpha\left(y\gamma,-\frac12y^2\gamma^2\right)\cdot u_\alpha(0,z\gamma)\\
&\qquad=u_\alpha\left(x+y\gamma, \frac12(x^2-y^2\gamma^2)+\gamma(z-xy)\right).
\end{align*}
Suppose that $\phi(x,y,z)=\phi(x',y',z')$, so that
\begin{align*}
x+y\gamma&=x'+y'\gamma,\\
\frac12(x^2-y^2\gamma^2)+\gamma(z-xy)&=\frac12(x'^2-y'^2\gamma^2)+\gamma(z'-x'y').
\end{align*}
Recall that $\{1,\gamma\}$ is an $\O_F$-basis of $\O_E$ (and in particular an $F$-basis of $E$); from this fact, it follows easily that $(x,y,z)=(x',y',z')$.  This proves injectivity.

For surjectivity, fix any $s,t\in E$ such that $s\ol s=t+\ol t$ which satisfy $v(s)>\frac1{p-1}-a$, $v(t)>\frac1{p-1}-2a$, so $u_\alpha(s,t)\in U_\alpha^a$ is an arbitrary element.  There exist $x,y\in F$ such that $s=x+y\gamma$, which by Lemma \ref{traceless} satisfy the equality
$$\min\{v(x),v(y)+v(\gamma)\}=v(s)>\frac1{p-1}-a.$$
Similarly, there are $w_0,z_0\in F$ such that $t=w_0+z_0\gamma$, which then satisfy
$$\min\{v(w_0), v(z_0)+v(\gamma)\}=v(t)>\frac1{p-1}-2a.$$
We then get directly from $s\ol s=t+\ol t$ that $w_0=\frac12(x^2-y^2\gamma^2)$.  Now, let $z\coloneqq z_0+xy$, and note from our above information on the valuations of our elements that
$$v(x)>\frac1{p-1}-a,\qquad v(y)>\frac1{p-1}-a-v(\gamma),\qquad\text{and } v(z_0)>\frac1{p-1}-2a-v(\gamma),$$
so 
\begin{align*}
v(z)&\geq \min\{v(z_0),v(x)+v(y)\}\\
&>\min\left\{\frac1{p-1}-2a-v(\gamma),\frac2{p-1}-2a-v(\gamma)\right\}\\
&=\frac1{p-1}-2a-v(\gamma).
\end{align*}
Thus, $(x,y,x)\in \mf m_F^{n_1}\times \mf m_F^{n_2}\times \mf m_F^{n_3}$, so we can conclude that
$$\phi(x,y,z)=u_\alpha\left(x+y\gamma, \frac12(x^2-y^2\gamma^2)+\gamma(z-xy)\right)=u_\alpha(s,t),$$
finishing the proof that $\phi:\mf m_F^{n_1}\times \mf m_F^{n_2}\times \mf m_F^{n_3}\to U_\alpha^a$ is bijective.  The multiplication map $U_{\alpha,1}^a\times U_{\alpha,2}^a\times U_{\alpha,3}^a\to U_\alpha^a$ must hence also be bijective.

Next, we must check that $\omega$ is compatible with $(U_{\alpha,1}^a,U_{\alpha,2}^a,U_{\alpha,3}^a)$, i.e., that
$$\omega(g_1g_2g_3)=\min\{\omega(g_1),\omega(g_3),\omega(g_3)\}$$
for any $g_i\in U_{\alpha,i}^a$.  Due to our bounds on the valuations of $(x,y,z)\in \mf m_F^{n_1}\times \mf m_F^{n_2}\times \mf m_F^{n_3}$, one can check that 
$$\omega\left(u_\alpha\left(x, \frac12x^2\right)\right)=v(x)+a,\qquad \omega\left(u_\alpha\left(y\gamma, \frac12y^2\gamma^2\right)\right)=v(y)+v(\gamma)+a,$$ 
$$\omega\left(u_\alpha(0,z\gamma)\right)=v(z)+v(\gamma)+2a.$$
Thus, our requirement translates to checking for all $(x,y,z)\in \mf m_F^{n_1}\times \mf m_F^{n_2}\times \mf m_F^{n_3}$ that
$$\omega(g)=\min\{v(x)+a, v(y)+v(\gamma)+a, v(z)+v(\gamma)+ 2a\},$$
where
$$g=\phi(x,y,z)=u_\alpha\left(x+y\gamma, \frac12(x^2-y^2\gamma^2)+\gamma(z-xy)\right).$$
Since $v(x+\gamma y)=\min\{v(x),v(y)+v(\gamma)\}$ by Lemma \ref{traceless}, this is equivalent to checking that
$$\omega(g)=\min\{a+v(x+y\gamma),2a+v(\gamma z)\}.$$
By definition, we already know that
$$\omega(g)=\min\left\{a+v(x+y\gamma),2a+v\left(\frac12(x^2-y^2\gamma^2)+\gamma(z-xy)\right)\right\}.$$
Assume that 
$$v(\gamma z)\neq v\left(\frac12(x^2-y^2\gamma^2)+\gamma(z-xy)\right),$$
which is the only possible situation where the above two minima could differ; we will prove in this case that $a+v(x+y\gamma)$ attains both minima, hence they are equal.  This assumption implies automatically that
$$v\left(\frac12(x^2-y^2\gamma^2)-xy\gamma\right)\leq v(\gamma z),$$
which we first claim implies $v(x+y\gamma)<a+v(\gamma z)$.  By Lemma \ref{traceless}, as $\gamma^2\in \O_F$, we know that
\begin{align*}
v\left(\frac12(x^2-y^2\gamma^2)-xy\gamma\right)&=\min\left\{v\left(\frac12(x^2-y^2\gamma^2)\right), v(xy\gamma)\right\}\\
&=\min\{2v(x+y\gamma), v(x)+v(y\gamma)\},
\end{align*}
where the last equality follows from $p\neq 2$ and $x^2-y^2\gamma^2=(x+y\gamma)(\ol{x+y\gamma})$.  If $2v(x+y\gamma)\leq v(\gamma z)$, then 
$$v(x+y\gamma)\leq -v(x+y\gamma)+v(\gamma z)<a-\frac1{p-1}+v(\gamma z)<a+v(\gamma z).$$
Similarly, if $v(x)+v(y\gamma)\leq v(\gamma z)$, then 
$$v(x)\leq -v(y\gamma)+v(\gamma z)<a-\frac1{p-1} +v(\gamma z)<a+v(\gamma z).$$
This would imply that
$$v(x+y\gamma)=\min\{v(x),v(y\gamma)\}\leq v(x)<a+v(\gamma z),$$
so we have in all cases from our assumption that $v(x+y\gamma)<a+v(\gamma z)$.  Thus, we need only verify under our assumption that $\omega(g)=a+v(x+y\gamma)$, i.e., that
$$v(x+y\gamma)\leq a+v\left(\frac12(x^2-y^2\gamma^2)+\gamma(z-xy)\right).$$
Again, by Lemma \ref{traceless}, the latter is given by
$$a+\min\left\{v\left(\frac12(x^2-y^2\gamma^2)\right), v(\gamma(z-xy)\right\}=a+\min\{2v(x+y\gamma), v(\gamma(z-xy))\},$$
so it is sufficient to show that $v(x+y\gamma)\leq a+ v(\gamma(z-xy))$.  We know $v(\gamma(z-xy))\geq \min\{v(\gamma z), v(xy\gamma)\}$ and $a+v(\gamma z)>v(x+y\gamma)$, so it is sufficient to check that $v(xy\gamma)+a>v(x+y\gamma)$.  Indeed, 
$$v(xy\gamma)=v(x)+v(y\gamma)>v(x)+\frac1{p-1}-a>v(x)-a$$
so
$$v(x+y\gamma)=\min\{v(x),v(y\gamma)\}<v(xy\gamma)+a$$
as claimed.  This completes the proof that $\omega$ is compatible with $(U_{\alpha,1}^a, U_{\alpha,2}^a, U_{\alpha,3}^a)$.
\end{proof}

\begin{rem} \label{compatible-saturation}
Suppose that $H$ is a $p$-valuable group, and $\omega$ is a $p$-valuation on $H$ compatible with a list of subgroups $(H_1,\ldots, H_s)$ as above.  If $(b_{i,1},\ldots, b_{i,t_i})$ is an ordered basis of $(H_i,\omega|_{H_i})$ for each $i\leq s$, then we can produce an ordered basis for $H$ given by 
$$(b_{1,1},\ldots, b_{1,t_1},\quad b_{2,1},\ldots, b_{2,t_2},\quad \ldots, \quad b_{s,1},\ldots, b_{s,t_s}).$$
In particular, if each $(H_i,\omega|_{H_i})$ is saturated, then $\omega(b_{i,j})\leq \frac p{p-1}$ for all $i\leq s$, $j\leq t_i$, from which it follows that also $(H,\omega)$ is saturated.
\end{rem}

\begin{cor}
The $p$-valuation $\omega$ on $U_\alpha^a$ is saturated.
\end{cor}

\begin{proof}
Recall that we have isomorphisms $\phi_j:\mf m_F^{n_k}\xto{\sim} U_{\alpha,j}^a$ for each $1\leq j\leq 3$, and
$$\omega(\phi_1(x))=v(x)+a,\qquad \omega(\phi_2(y))=v(y)+v(\gamma)+a,\qquad \omega(\phi_3(z))=v(z)+v(\gamma)+2a$$
for all $(x,y,z)\in \mf m_F^{n_1}\times \mf m_F^{n_2}\times\mf m_F^{n_3}$.  Recall that $n_1,n_2,n_3$ denote the least integers strictly greater than 
$$e_F\cdot \left(\frac1{p-1}-a\right),\qquad e_F\cdot \left(\frac1{p-1}-a-v(\gamma)\right),\qquad e_F\cdot\left( \frac1{p-1}-2a-v(\gamma)\right),$$
respectively.  Thus, it is immediate from three applications of Lemma \ref{OKsaturation} that $\omega$ restricts to a saturated $p$-valuation on $U_{\alpha, j}^a$ for each $1\leq j \leq 3$, and hence is saturated on $U_\alpha^a$.
\end{proof}

We have now completed the proof of Proposition \ref{root-group-pval}.
We note the following easy corollary of Proposition \ref{compatibility}.

\begin{cor} \label{basis}
Let $(b_1,\ldots, b_\ell)$ be an ordered basis for $\O_F$, and fix a traceless element $\gamma\in \O_E$ such that $\{1,\gamma\}$ is an $\O_F$-basis for $\O_E$.  Fix any uniformizer $\varpi_F\in \O_F$.  Then the following is an ordered basis for $U_\alpha^a$:
\begin{align}
g_{1,i}&\coloneqq u_\alpha\left(\varpi_F^{n_1}b_i, \frac12\varpi_F^{2n_1}b_i^2\right),\notag\\
g_{2,i}&\coloneqq u_\alpha\left(\varpi_F^{n_2}b_i\gamma, -\frac12\varpi_F^{2n_2}b_i^2\gamma^2\right),\tag{$i\leq \ell$}\\
g_{3,i}&\coloneqq u_\alpha(0,\varpi_F^{n_3}b_i\gamma).\notag
\end{align}
\end{cor}

\begin{proof}
Since our isomorphisms $\phi_j$ between $\mf m_F^{n_j}$ and each of $U_{\alpha, j}^a$ pull $\omega$ back to mere shifts of the standard $p$-valuation on $\mf m_F^{n_j}$, i.e., we have that
$$\omega(\phi_1(x))=v(x)+a,\qquad \omega(\phi_2(y))=v(y)+v(\gamma)+a,\qquad \omega(\phi_3(z))=v(z)+v(\gamma)+2a,$$
then it follows from Remark \ref{OKbasis} that $(\varpi_F^{n_j}b_1,\ldots, \varpi_F^{n_j}b_\ell)$ is an ordered basis for $\mf m_F^{n_j}$ with respect to the $p$-valuations induced by these three isomorphisms.  Thus, their images 
$$(\phi_j(\varpi_F^{n_j}b_1),\ldots, \phi_j(\varpi_F^{n_j}b_\ell))$$
form an ordered basis of $(U_{\alpha,j}^a,\omega)$ for each $1\leq j\leq 3$, and we have that $\phi_j(\varpi_F^{n_j}b_i)=g_{j,i}$ for $1\leq j\leq 3$, $1\leq i\leq \ell$.  By our previous remark, we can thus concatenate to produce the claimed ordered basis of $U_\alpha^a$.
\end{proof}


\section{\texorpdfstring{$p$}{p}-valuations on tori}

Let $K/\Q_p$ be a finite extension equipped with the usual valuation $v$ normalized by $v(p)=1$, and set $e_K\coloneqq e(K/\Q_p)$.  Recall that if $r\in \Z^+$, we have the group of $r$-units $U_K^{(r)}=1+\mf m_K^r$.


\begin{rem} \label{Gm-log-exp}
It is a standard result that if $r$ is the smallest integer greater than $e_K/(p-1)$, and $n\geq r$, then $U_K^{(n)}=1+\mf m_K^n$ is saturated with $p$-valuation $\omega(x)=v(x-1)+\frac{r-n}{e_K}$.  This follows from the fact that $\log$ and $\exp$ give isomorphisms $U_K^{(m)}\iso \mf m_K^m$ for any $m\geq r$, and $\mf m_K^n$ has saturated $p$-valuation 
$$\omega_0(x)=v(x)+\frac{r-n}{e_K}$$
for any $n$ due to Lemma \ref{OKsaturation}.
\end{rem}

\begin{note}
In particular, the above shows that $U_K^{(r)}$ for $r$ the least integer greater than $e_K/(p-1)$ is saturated by
$$\omega(x)=v(x-1).$$
This is also equal to
$$\omega(x)=\frac1{e_K}\cdot \sup\left\{n\in \Z^+\Bigst  x\in U_K^{(n)}\right\}.$$
\end{note}


Let now $T$ be an algebraic $K$-torus.  We have the following standard definition (see \cite[pp. 86 \& 669]{KP}, for example).

\begin{defn}
The torus $T$ is called \textbf{induced} if there is some $m\in \Z^+$ and there are finite extensions $L_i/K$ for $i\leq m$ such that
$$T\iso \prod_{i=1}^m {R}_{L_i/K}\G_m$$
are isomorphic $K$-tori.

We also say that $T$ is \textbf{weakly induced} if there is a tamely ramified extension $L/K$ such that the base change $T_L$ is an induced torus.
\end{defn}

\begin{rem}
$T$ is weakly induced if and only if there is a basis of $X^*(T)$ on which the wild inertia subgroup of $\Gal(\ol K/K)$ acts via a permutation action.  (See \cite[p. 669]{KP}.)
\end{rem}

\begin{rem}
Let $G$ be a quasisplit connected reductive $K$-group, and $T\leq G$ a maximal $K$-torus contained in a Borel $K$-subgroup.  If $G$ is semisimple, as well as either adjoint or simply connected, then $T$ is an induced torus.  (See \cite[p. 86]{KP}.)
\end{rem}

\begin{notn}
If $T$ is a weakly induced torus, the filtration we use on it is the \textbf{standard filtration} given by $T(K)_0$ the Iwahori subgroup $T(K)^0$, and
$$T(K)_r\coloneqq\{t\in T(K)\st  v(\chi(t)-1)\geq r\text{ for all } \chi\in X^*(T)\}$$
for $r>0$.  The maximal pro-$p$ subgroup of $T(K)$ is given by
$$T(K)_{0+}\coloneqq\{t\in T(K)\st v(\chi(t)-1)>0\text{ for all }\chi\in X^*(T)\}.$$
(For more details, see \cite[pp. 245 \& 692]{KP}.)
\end{notn}

\begin{note}
The appropriate filtration on tori to be using is the minimal congruent filtration introduced by Yu in \cite{yu-smooth_models}.  However, when $T$ is weakly induced, this coincides with the standard filtration by \cite[Corollary B.10.13]{KP}, and is consequently much easier to study; we will work almost exclusively with this case.
\end{note}

\begin{ex}
Say $T=\prod_{i=1}^m{R}_{L_i/K}\G_m$ (so $T$ is induced).  Then we have the equalities
$$T(K)=\prod_{i=1}^m L_i^\times,\qquad T(K)_0=\prod_{i=1}^m \O_{L_i}^\times,\qquad T(K)_{0+}=\prod_{i=1}^m U_{L_i}^{(1)}.$$
\end{ex}


The following is a standard result; the proof is presented for convenience.

\begin{lem} \label{split-torus-sat}
Let $T$ be a split $K$-torus.  Then
$$T(K)_{\frac1{p-1}+}=\{t\in T(K)\st v(\chi(t)-1)>\frac1{p-1} \text{ for all }\chi\in X^*(T)\}$$
has a natural $p$-valuation given by $\omega(t)=\sup\left\{r\in \left(\frac1{p-1},\infty\right)\Bigst t\in T(K)_r\right\}$.  Moreover, $(T(K)_{\frac1{p-1}+},\omega)$ is saturated.
\end{lem}

\begin{proof}
Since $T$ is split, then $T(K)\iso (K^\times)^m$ for some $m$.  Now $T(K)_{\frac1{p-1}+}\iso(U_K^{(r)})^m$ where $r$ is the least integer greater than $e_K/(p-1)$, and the function $\omega$ on $T(K)_{\frac1{p-1}+}$ corresponds to the function given on $(x_i)_{i\leq n}\in (U_K^{(r)})^m$ by 
$$\min_{i\leq m} \frac1{e_K}\cdot \sup\{n\in \Z^+\st x_i\in U_K^{(n)}\}=\min_{i\leq m} v(x_i-1).$$
Since $\omega_i(x_i)\coloneqq v(x_i-1)$ is saturated on $U_K^{(r)}$ for each $i$, then it follows easily that $((U_K^{(r)})^m,\omega)$, and hence $(T(K)_{\frac1{p-1}+},\omega)$, is also saturated.
\end{proof}

\begin{cor} \label{torus-saturation}
Let $T$ be a weakly induced $K$-torus.  Then $T(K)_{\frac1{p-1}+}$ has a natural $p$-valuation given by
$$\omega(t)=\sup\left\{r\in \left(\frac1{p-1},\infty\right)\Biggst t\in T(K)_r\right\}.$$
Moreover, $\left(T(K)_{\frac1{p-1}+},\omega\right)$ is $p$-saturated.
\end{cor}

\begin{proof}
Let $E/K$ be a finite Galois extension such that the base change $T_E$ is split.  Note that
$$T(K)_r=T(E)_r\cap T(K)=T_E(E)_r^{\Gal(E/K)}$$
for all $r>0$ (see \cite[Remark B.6.8]{KP} or even \cite[Proposition 12.9.4]{KP}), and similarly $T(K)_{\frac1{p-1}+}=T_E(E)_{\frac1{p-1}+}^{\Gal(E/K)}$.  Thus, the $p$-valuation we have defined on $T(K)_{\frac1{p-1}+}$ is merely the restriction of the one we have defined previously on the split $T_E(E)_{\frac1{p-1}+}$.  This result is now immediate from Lemma \ref{split-torus-sat} and our earlier observation in Proposition \ref{invariants} that if a group is saturated, its subgroup of invariants under a family of homomorphisms is also saturated.
\end{proof}

\section{On the Moy-Prasad filtration} \label{moy-prasad}

In this section, we will prove the following result.  Assume $p>2$.  Let $K/\Q_p$ be any finite extension, and let $G$ be a connected reductive $K$-group.  Fix a maximally $K$-split maximal $K$-torus $T\leq G$, and assume that $T$ is weakly induced.  Let $\mc B(G)$ denote the semisimple Bruhat-Tits building of $G$, and for $r\geq 0$, let $G_{x,r}$ denote the $r$th term of the Moy-Prasad filtration on $G(K)$ associated to any $x\in \mc B(G)$, which is the closed subgroup generated by $T(K)_r$ and $U_{\alpha, x, r}$ for all $\alpha\in \Phi^{\mathrm{nd}}$ (see \cite[Definition 7.3.3]{KP}).  We then also write
$$G_{x,r+}\coloneqq \bigcup_{s>r} G_{x,s}$$
for any $r\geq 0$, as usual.

\begin{thm}
For any $x\in \mc B(G)$, there is a saturated $p$-valuation $\omega$ on $G_{x,\frac1{p-1}+}$ defined by
$$\omega(g)=\min\left\{r>\frac1{p-1}\Biggst g\in G_{x,r}\right\}.$$
\end{thm}

To prove this, we will first consider the quasisplit case, and then proceed via unramified descent to show the general case.

\subsection{The quasisplit case}\label{qsplit-case}

In this section, we will mimic the style of argument used in \cite[Section 3]{LS}.  As always, we assume $p>2$.

Suppose that $G$ is a quasisplit connected reductive $K$-group for $K/\Q_p$ a finite extension.  Fix any $x\in \mc B(G)$.  Let $T\leq G$ be a maximally $K$-split maximal $K$-torus, $S\leq T$ its maximal $K$-split subtorus, and let $B\leq G$ be a Borel $K$-subgroup containing $T$.  Recall that we have fixed an absolutely special vertex $x_0$ of the semisimple apartment $\mc A(S)$ associated to $S$ arising from our Chevalley-Steinberg system $X$; we may indeed assume without loss of generality that $x$ lies in the apartment $A(S)$, as we could merely change our tori accordingly.  Moreover, we may also assume without loss of generality that $x$ lies in the dominant base chamber $C\sub \mc A(S)$ associated to $B$ and $x_0$.

We now recall the following fact from \cite[Lemma 3.1]{LS}:

\begin{lem} \label{congruence-pval}
Let $E/\Q_p$ be a finite extension with $e_E=e(E/\Q_p)$, and suppose $\mc G$ is an $\O_E$-subgroup scheme of $\GL_N$ for some $N\in \Z^+$.  For $r\in \Z^+$, consider the congruence subgroup $K_r$ defined by
$$K_r\coloneqq \ker\left(\mc G(\O_E)\to \mc G(\O_E/\mf m_E^r)\right).$$
Then for $r>\frac {e_E}{p-1}$, the expression
$$\omega(g)=\frac1{e_E}\cdot \sup\{n\in \Z^+\st g\in K_n\}$$
defines a $p$-valuation on $K_r$.
\end{lem}

As in \cite[Section 3]{LS} and \cite[Section 5]{derivedinduction}, we apply this to the following $\mc G$.  Assume that $E/K$ is a yet unspecified extension for which $G_E$ is split, so that $T_E$ is a split torus.  Then our absolutely special vertex $x_0$ has hyperspecial image in the apartment associated to $T_E$ in the semisimple Bruhat-Tits building of $G_E$ (see \cite[Section 7.11]{KP}).  Let $\mc G=\mc G_{x_0}^0$ denote the Bruhat-Tits parahoric group scheme over $\mc O_E$ associated to the image of $x_0$ in $\mc B(G_E)$ (see \cite[Axiom 4.1.20]{KP}).

Our fixed Chevalley-Steinberg system $X$ induces a filtration on the absolute root groups (relative to $E$) by
$$U_{\tilde\alpha,x_0,r}^E\coloneqq \{u_{\tilde\alpha}(x) \st x\in E, v(x)\geq r\}$$
for any $\tilde\alpha\in \tilde\Phi$ and any $r\geq 0$.

Let $K_r$ denote the $r$th congruence subgroup of $\mc G(\O_E)$, so $K_r$ is generated by $U_{\tilde\alpha, r/e_E}^E$ for $\tilde\alpha\in \tilde\Phi$ and $T(E)_{r/e_E}$.  More precisely, by \cite[Lemma 5.1]{derivedinduction}, we have that $K_r=(G_E)_{x_0,r/e_E}$ is the $(r/e_E)$th term in the Moy-Prasad filtration of $G_E(E)$ associated to $x_0\in \mc B(G_E)$, and the multiplication map
$$\prod_{\tilde\alpha\in \tilde\Phi^-} U_{\tilde\alpha, x_0,r/e_E}^E\times T(E)_{r/e_E}\times \prod_{\tilde\alpha\in \tilde\Phi^+} U_{\tilde\alpha, x_0,r/e_E}^E\to K_r$$
is a homeomorphism for any $r\in \Z^+$.

We now consider the Moy-Prasad filtration group $G_{x,\frac1{p-1}+}$.  We have an Iwahori factorization by \cite[Proposition 13.2.5]{KP}; that is, the multiplication map gives a homeomorphism
$$\prod_{\alpha\in \Phi^{-,\mathrm{nd}}} U_{\alpha,x,\frac1{p-1}+}\times T(K)_{\frac1{p-1}+}\times\prod_{\alpha\in \Phi^{+,\mathrm{nd}}}U_{\alpha,x,\frac1{p-1}+}\to G_{x,\frac1{p-1}+},$$
where the factors in each product are taken in an arbitrary but fixed order.

Recall that we can write $x=x_0+v$ for a unique $v\in V(S\ad)=X_*(S\ad)\otimes_\Z \R$, which as we are assuming $x$ lies in the dominant base chamber $C$ means that $v$ is dominant.  To simplify notation, we will write $M_\Q\coloneqq M\otimes_\Z \Q$ for any $\Z$-module $M$.

\begin{lem} \label{conjlemma}
Suppose that $G$ is quasisplit and $T$ is weakly induced.  Assume that $v\in X_*(S\ad)_\Q$.  Then there exists a pair $(E,t)$ consisting of a finite extension $E/K$ such that $G_E$ splits and an element $t\in T(E)$ such that $tG_{x,\frac1{p-1}+}t\inv\sub K_r$, where $r$ is the smallest integer $>e_E/(p-1)$ for $e_E=e(E/\Q_p)$.
\end{lem}

\begin{proof}
This proof is directly inspired by that of \cite[Lemma 3.2]{LS}.  Pick any lift $\mu_0\in X_*(S)_\Q$ of $v$, so there exists some $n\in \Z^+$ such that $\mu\coloneqq n\mu_0\in X_*(S)=X_*(T)^\Theta$ by clearing denominators.  Pick now $E/K$ large enough such that $G_E$ splits and there is some $c\in E^\times$ with $v(c)=\frac 1n$.  We claim that $t\coloneqq \mu(c)\in T(E)$ suffices.

Indeed, recall that the multiplication map gives a homeomorphism
$$\prod_{\tilde\alpha\in \tilde\Phi^-} U_{\tilde\alpha, x_0,r/e_E}^E\times T(E)_{r/e_E}\times \prod_{\tilde\alpha\in \tilde\Phi^+} U_{\tilde\alpha, x_0,r/e_E}^E\to K_r$$
for any $r\geq 0$.  For each $\alpha\in \Phi$, choose $\tilde\alpha\in \tilde\Phi$ lying above $\alpha$, and if $\alpha$ is of type R2, we also pick $\tilde\alpha'$ above $\alpha$ and $\tilde\beta$ above $2\alpha$ exactly as in Remark \ref{Chev-Stein-isoms}.

Using our Iwahori factorization
$$\prod_{\alpha\in \Phi^{-,\mathrm{nd}}} U_{\alpha,x,\frac1{p-1}+}\times T(K)_{\frac1{p-1}+}\times\prod_{\alpha\in \Phi^{+,\mathrm{nd}}}U_{\alpha,x,\frac1{p-1}+}\to G_{x,\frac1{p-1}+},$$
and the analogous one above for $K_r$, we see that $tG_{x,\frac1{p-1}+}t\inv\sub K_r$ can be checked on each Iwahori factor.
\begin{enumerate}[{(i)}]
\item For $\alpha\in \Phi^\mathrm{nd}$ of type R1, 
\begin{align*}
tU_{\alpha,x,\frac1{p-1}+}t\inv\sub \prod_{\sigma\in \Theta/\Theta_{\tilde\alpha}}U_{\sigma\tilde\alpha,x_0,r/e_E}^E
&\iff (\sigma\tilde\alpha)(t)\mf m_{K_{\tilde\alpha}}^{m_\alpha}\sub \mf m_E^r\qquad \forall \sigma\in \Theta
\end{align*}
where $m_\alpha\in \Z$ is the least integer strictly greater than 
$$e_{K_{\tilde\alpha}}\cdot \left(\frac1{p-1}-\langle \alpha, v\rangle\right).$$
Since $t=\mu(c)$, and $\mu_0=\frac1n \mu$ is Galois-invariant, note that 
\begin{align*}
v((\sigma\tilde\alpha)(t))=v((\sigma\tilde\alpha)(\mu(c)))&=v(c^{\langle \sigma\tilde\alpha, \mu\rangle})=\langle \sigma\tilde\alpha, \mu\rangle \cdot v(c)\\
&=\langle \sigma\tilde\alpha, \mu_0\rangle=\langle \tilde\alpha, \sigma\inv\mu_0\rangle=\langle \tilde\alpha, \mu_0\rangle=\langle \alpha, v\rangle.
\end{align*}
Thus, $(\sigma\tilde\alpha)(t)m_{K_{\tilde\alpha}}^{m_\alpha}\sub m_E^r$ is immediate from the description of $m_\alpha$ and $r$ being the least integers strictly greater than 
$$e_{K_{\tilde\alpha}}\cdot \left(\frac1{p-1}-\langle \alpha, v\rangle\right)\qquad\mathrm{and}\qquad e_E\cdot \frac1{p-1},$$
respectively.
\item For $\alpha\in \Phi^\mathrm{nd}$ of type R2,
\begin{align*}
tU_{\alpha,x,\frac1{p-1}+}t\inv&\sub \prod_{\sigma\in \Theta/\Theta_{\tilde\beta}}U_{\sigma\tilde\alpha,x_0,r/e_E}^E\cdot U_{\sigma\tilde\alpha',x_0,r/e_E}^E\cdot U_{\sigma\tilde\beta,x_0,r/e_E}^E\qquad\\
\iff&(\sigma\tilde\alpha)(t)\mf m_{K_{\tilde\alpha}}^{m_\alpha}, \quad (\sigma\tilde\alpha')(t)\mf m_{K_{\tilde\alpha}}^{m_\alpha}, \quad (\sigma\tilde\beta)(t)\mf m_{K_{\tilde\alpha}}^{n_\alpha}\quad \sub \mf m_E^r\qquad \forall \sigma\in \Theta,
\end{align*}
where $m_\alpha$ and $n_\alpha$ are the least integers strictly greater than 
$$e_{K_{\tilde\alpha}}\cdot \left(\frac1{p-1}-\langle \alpha, v\rangle\right)\qquad\mathrm{and}\qquad e_{K_{\tilde\alpha}}\cdot \left(\frac1{p-1}-\langle 2\alpha, v\rangle\right),$$
respectively.  Once again, we can check from $t=\mu(c)$ that
$$v((\sigma\tilde\alpha)(t))=v((\sigma\tilde\alpha')(t))=\langle \alpha, v\rangle$$
and
$$v((\sigma\tilde\beta)(t))=\langle 2\alpha, v\rangle,$$
so as in the last item, our requirement that
$$(\sigma\tilde\alpha)(t)\mf m_{K_{\tilde\alpha}}^{m_\alpha}, \quad (\sigma\tilde\alpha')(t)\mf m_{K_{\tilde\alpha}}^{m_\alpha}, \quad (\sigma\tilde\beta)(t)\mf m_{K_{\tilde\alpha}}^{n_\alpha}\quad \sub \mf m_E^r\qquad \forall \sigma\in \Theta$$
is immediately satisfied as $r$ is the least integer strictly greater than $e_E\cdot \frac1{p-1}$.
\item For the torus, since $T(E)$ is abelian and $T(K)_{\frac1{p-1}+}\sub T(E)_{\frac1{p-1}+}$, it is immediate that
$$tT(K)_{\frac1{p-1}+}t\inv=T(K)_{\frac1{p-1}+}\sub T(E)_{r/e_E}.$$
\end{enumerate}
Thus, $t=\mu(c)$ indeed satisfies the requirement that $tG_{x,\frac1{p-1}+}t\inv \sub K_r$.
\end{proof}

\begin{rem}
Whenever $v\in X_*(S\ad)_\Q$, this result gives us a $p$-valuation $\omega$ on $G_{x,\frac1{p-1}+}$ by restricting and conjugating the one from $K_r$ for $r$ in Lemma \ref{congruence-pval}.
\end{rem}

Before proving the next result, we need a technical lemma.

\begin{lem} \label{density}
For each $x\in \mc A(S)$, there is a subset $Q$ of 
$$x_0+X_*(S\ad)_\Q\sub \mc A(S)$$
either containing $x$ or which has $x$ as a limit point, and which satisfies the requirement that $G_{x,\frac1{p-1}+}=G_{x',\frac1{p-1}+}$ for all $x'\in Q$.
\end{lem}

\begin{proof}
Using Iwahori factorizations, to conclude that $G_{x,\frac1{p-1}+}=G_{x',\frac1{p-1}+}$, it is sufficient to show that $U_{\alpha, x, \frac1{p-1}+}=U_{\alpha, x', \frac1{p-1}+}$ for each $\alpha\in \Phi^{\mathrm{nd}}$; equivalently, we need $v'\coloneqq x'-x_0\in X_*(S\ad)_\Q$ to satisfy
$$U_{\alpha, x_0+v, \frac1{p-1}+}=U_{\alpha, x_0+v', \frac1{p-1}+}$$
for all $\alpha\in \Phi^{\mathrm{nd}}$, where $v=x-x_0$.  Using our usual description of the root groups when written in this form (see Subsection \ref{R1} and Remark \ref{R2-desc}), it is equivalent to show that
$$\left\lfloor e_{K_{\tilde\alpha}}\cdot \left(\frac1{p-1}-\langle \alpha, v\rangle\right)\right\rfloor=\left\lfloor e_{K_{\tilde\alpha}}\cdot \left(\frac1{p-1}-\langle \alpha, v'\rangle\right)\right\rfloor$$
for all $\alpha\in \Phi^{\mathrm{nd}}$, and that 
$$\left\lfloor e_{K_{\tilde\alpha}}\cdot \left(\frac1{p-1}-\langle 2\alpha, v\rangle\right)\right\rfloor=\left\lfloor e_{K_{\tilde\alpha}}\cdot \left(\frac1{p-1}-\langle 2\alpha, v'\rangle\right)\right\rfloor$$
when $\alpha\in \Phi^{\mathrm{nd}}$ is of type R2.  (Here, $\lfloor s \rfloor$ denotes the greatest integer $n$ with $n\leq s$ for any $s\in \R$.)  Since $\langle -\alpha, v'\rangle= -\langle \alpha, v'\rangle$ and $\langle \pm 2\alpha, v'\rangle=\pm2 \langle \alpha, v'\rangle$, we can collect all of these conditions to have the following concise description of a condition on $v'$ for each $\alpha\in \Phi^{+,\mathrm{nd}}$ which is sufficient to ensure $G_{x,\frac1{p-1}+}=G_{x',\frac1{p-1}+}$:
\begin{enumerate}
\item If $\langle \alpha, v\rangle$ is already rational, we require that $\langle \alpha, v'\rangle=\langle \alpha, v\rangle$.
\item If $\langle \alpha, v\rangle$ is not rational, then there is an open set of values $V_\alpha$ containing $\langle \alpha, v\rangle$ which we require $\langle \alpha, v'\rangle$ to lie in.
\end{enumerate}
If every $\langle \alpha, v\rangle$ is already rational, then $v\in X_*(S\ad)_\Q$, so we can just take $Q=\{x\}$; assume without loss of generality that this is not the case.  Note that each $\langle \alpha, -\rangle$ is a rational linear functional $X_*(S\ad)_\Q\to \Q$, so our argument will boil down to mere linear algebra.

First, let $A$ denote the set of $v'\in X_*(S\ad)_\Q$ satisfying the condition that $\langle \alpha, v'\rangle=\langle \alpha, v\rangle$ for each $\alpha\in \Phi^{+,\mathrm{nd}}$ which has $\langle \alpha, v\rangle\in \Q$.  Since each $\langle \alpha, -\rangle$ is a rational linear functional, and these conditions have a mutual solution $v$ in $V(S\ad)=X_*(S\ad)_\Q\otimes_\Q \R$, it is an elementary linear algebra exercise to show that there is at least one solution in $X_*(S\ad)_\Q$.  Thus, $A$ is an affine subspace of $X_*(S\ad)_\Q$ over some subspace $W\sub X_*(S\ad)_\Q$.  Write $A_\R$ for the unique affine subspace of $V(S\ad)$ over $W\otimes_\Q \R$ satisfying $A_\R\cap X_*(S\ad)_\Q=A$, which is easily seen to contain $v$.

We now have, for each $\alpha\in \Phi^{+,\mathrm{nd}}$ such that $\langle \alpha, v\rangle$ is not already rational, an affine functional $\langle \alpha, -\rangle: A \to \Q$ which we require to attain values within an open set $V_\alpha\sub \R$ which contains $\langle \alpha, v\rangle$.  These induce real affine functionals $\langle \alpha, -\rangle: A_\R\to \R$; let $V$ denote the open subset of $v'\in A_\R$ satisfying our condition that $\langle \alpha, v'\rangle\in V_\alpha$ for each such $\alpha$, which is nonempty as it contains $v$ itself.  By the density of $\Q$ in $\R$, it is immediate that $V\cap A=V\cap X_*(S\ad)_\Q\sub V(S\ad)$ is nonempty and has $v$ as a limit point.  Thus, if we write $Q=x_0+(V\cap A)$, we find that $Q\sub x_0+X_*(S\ad)_\Q$ is a subset with $x_0$ as a limit point satisfying the requirement that $G_{x,r}=G_{x',r}$ for all $x'\in Q$.
\end{proof}

\begin{prop} \label{existence}
Suppose that $G$ is quasisplit and $T$ is weakly induced.  Then for each $x\in \mc B(G)$, there exists a unique $p$-valuation $\omega_x$ on $G_{x,\frac1{p-1}+}$ satisfying the following properties:
\begin{enumerate}[\normalfont(a)]
\item $\omega$ is compatible with the Iwahori factorization
$$\prod_{\alpha\in \Phi^{-,\mathrm{nd}}}U_{\alpha,x,\frac1{p-1}+}\times T(K)_{\frac1{p-1}+}\times\prod_{\alpha\in \Phi^{+,\mathrm{nd}}}U_{\alpha,x,\frac1{p-1}+}\xto{\sim} G_{x,\frac1{p-1}+},$$
where the products are each taken in any (fixed) order.
\item If $\alpha\in \Phi$ is of type R1, then $\omega_x(u_\alpha(x))=v(x)+\langle \alpha,v\rangle$, where $x\in \mf m_{K_{\tilde\alpha}}^{m_\alpha}$ for $m_\alpha$ the least integer strictly greater than 
$$e_{K_{\tilde\alpha}}\cdot \left(\frac1{p-1}-\langle \alpha, v\rangle\right).$$
\item If $\alpha\in \Phi$ is of type R2, then
$$\omega_x(u_\alpha(x,y))=\min\{v(x)+\langle \alpha, v\rangle, v(y)+2\langle \alpha,v\rangle\},$$
where $x\in \mf m_{K_{\tilde\alpha}}^{m_\alpha}$ and $y\in \mf m_{K_{\tilde\alpha}}^{n_\alpha}$ satisfy $x\ol x=y+\ol y$ for $m_\alpha, n_\alpha$ the least integers strictly greater than 
$$e_{K_{\tilde\alpha}}\cdot \left(\frac1{p-1}-\langle \alpha, v\rangle\right)\qquad\mathrm{and}\qquad e_{K_{\tilde\alpha}}\cdot \left(\frac1{p-1}-\langle 2\alpha, v\rangle\right),$$
respectively.
\item If $t\in T(K)_{\frac1{p-1}+}$, then 
$$\omega_x(t)=\sup\{r\in \R^+\st t\in T(K)_r\}=\min\{v\big(\chi(t)-1\big)\st \chi\in X^*(T)\}.$$
\end{enumerate}
Moreover, this $p$-valuation is saturated, and given by the Moy-Prasad filtration associated to the point $x\in \mc B(G)$ via
$$\omega_x(g)=\sup\{r>0\st g\in G_{x,r}\}$$
for all $g\in G_{x,\frac1{p-1}+}$.
\end{prop}

\begin{proof}
This is a generalization of \cite[Proposition 3.5]{LS}.  Uniqueness is clear, as any $p$-valuation compatible with the Iwahori factorization will be determined by its restriction on each factor, where its values are determines by properties (b) and (c).  

Existence is similarly easy: suppose first that $v\in X_*(S\ad)_\Q$, and fix $(E,t)$ for $t=\mu(c)$ as in the proof of Lemma \ref{conjlemma}.  Then the $p$-valuation on $G_{x,\frac1{p-1}+}$ is defined by 
$$\omega_x(g)=\frac1{e_E}\cdot \sup\{n\in \Z^+\st tgt\inv \in K_n\},$$
and we obtain properties (a), (b), (c), and (d) by comparing the Iwahori factorization of $I$ and the analogous factorization of $K_n$, exactly as in the proof of Lemma \ref{conjlemma}.  We can then allow the case where $v\in V(S\ad)$ does not lie in $X_*(S\ad)_\Q$ by a simple application of Lemma \ref{limits}.  Indeed, use Lemma \ref{density} to construct a sequence of points $x_n\in \mc A(S)$ converging to $x$ which satisfy the requirements that $G_{x_n,\frac1{p-1}+}=G_{x,\frac1{p-1}+}$ and $v_n\coloneqq x_n-x_0\in X*(S\ad)_\Q$ for every $n$.  Then the $p$-valuations $\omega_{x_n}$ converge pointwise to the function $\omega_x$, which is bounded away from $\frac1{p-1}$, proving that $\omega_x$ is a $p$-valuation by Lemma \ref{limits}.

To see that $\omega_x$ is also given by the characterization using the Moy-Prasad filtration, we note that by \cite[Proposition 13.2.5]{KP} there is for all $r>0$ an Iwahori factorization on $G_{x,r}$.  Specifically, the product map
$$\prod_{\alpha\in \Phi^{-,\mathrm{nd}}} U_{\alpha, x, r}\times T(K)_r\times \prod_{\alpha\in \Phi^{+,\mathrm{nd}}} U_{\alpha, x, r}\to G_{x,r}$$
is a bijection.  It is now sufficient to show that the $p$-valuation $\omega_x$ agrees with the Moy-Prasad filtration on each Iwahori factor.  By construction, the two filtrations on the tori are equal, and the proof of Proposition \ref{root-group-pval} shows that the described filtrations on the root groups are precisely those coming from the Moy-Prasad filtration, as claimed.

Finally, as $\omega_x$ is compatible with the Iwahori factorization, it is enough to show that $\omega_x$ is saturated on each Iwahori factor to conclude that it is saturated on $G_{x,\frac1{p-1}+}$ by Remark \ref{compatible-saturation}.  Corollary \ref{torus-saturation} shows that $\omega_x$ is saturated on $T(K)_{\frac1{p-1}+}$, and Proposition \ref{root-group-pval} shows that it is saturated on $U_{\alpha, x, \frac1{p-1}+}$ for each $\alpha\in \Phi^{\mathrm{nd}}$.
\end{proof}

\subsection{The general case} \label{general-case}

Let now $G$ be any connected reductive $K$-group for $K/\Q_p$ a finite extension.  As always, assume $p>2$.  Assume that $G$ has a weakly induced maximally $K$-split maximal $K$-torus $T\leq G$, let $S\leq T$ be a maximal $K$-split torus, and let $Z\leq G$ be the centralizer of $S$ in $G$.  In \cite[Section 9.8]{KP}, a filtration on $Z(K)$ is constructed such that we have an Iwahori decomposition 
\begin{equation}\label{IwaDecomp}\prod_{\alpha\in \Phi^{-,\mathrm{nd}}} U_{\alpha,x,r}\times Z(K)_r\times\prod_{\alpha\in \Phi^{+,\mathrm{nd}}}U_{\alpha,x,r}\xto{\sim} G_{x,r}\end{equation}
for any $x\in \mc B(G)$ and any $r>0$ by \cite[Proposition 13.2.5(3)]{KP}.

We then have the following generalization of Proposition \ref{existence}.

\begin{thm} \label{general-existence}
For each $x\in \mc B(G)$, the filtration on $G_{x,\frac1{p-1}+}$ defined by
\begin{equation}\label{pval-def}\omega_x(g)\coloneqq \sup\{r>0\st g\in G_{x,r}\}\end{equation}
is a saturated $p$-valuation.  Moreover, this is compatible with the Iwahori decomposition \ref{IwaDecomp}.
\end{thm}

\begin{proof}
By a result of Steinberg in \cite{Steinberg}, there is a finite unramified (hence Galois) extension $L/K$ such that $G_L$ is a quasisplit reductive $L$-group (this fact is also used in \cite[p. 218]{totaro}).  Let $\Theta=\Gal(L/K)$.  By \cite[Proposition 9.8.3]{KP}, we have for any $r>0$ that 
\begin{equation*}
G_{x,r}=(G_L)_{x,r}\cap G(K)=(G_L)_{x,r}^\Theta,
\end{equation*}
so $\omega_x:G_{x,\frac1{p-1}+}\setminus\{1\}\to \left(\frac1{p-1},\infty\right)$ is merely the restriction to $G_{x,\frac1{p-1}+}=(G_L)_{x,\frac1{p-1}+}^\Theta$ of the saturated $p$-valuation $\omega_x:(G_L)_{x,\frac1{p-1}+}\setminus\{1\}\to \left(\frac1{p-1},\infty\right)$.  Thus, by Remark \ref{gen_invariants}, we find that $\omega_x$ is a saturated $p$-valuation on $G_{x,\frac1{p-1}+}$.  It is now immediate from the definition \ref{pval-def} that $\omega_x$ is compatible with the Iwahori decomposition \ref{IwaDecomp}.
\end{proof}

We record the following immediate group-theoretic consequences.  Part \ref{uniform} should be compared with \cite[Corollary C.4]{koziol-appendix} in the appendix authored by Kozioł and Schwein.

\begin{cor}
Fix any $x\in \mc B(G)$.
\begin{enumerate}[\normalfont(1)]
\item $G_{x,\frac1{p-1}+}$ is a $p$-torsion free pro-$p$ group.
\item For any $r>\frac1{p-1}$, $G_{x,r+1}=\{g^p\st g\in G_{x,r}\}$.
\item \label{uniform} For any $r\geq 1$, $G_{x,r}$ is a uniform pro-$p$ group.
\end{enumerate}
\end{cor}

\begin{proof}
These statements all are easy to prove for any finite rank saturated pro-$p$ group $(G,\omega)$, and are not unique to our situation.  Fact \ref{uniform} is perhaps the least standard, so we will prove that one for the convenience of the reader.  Because $p$-valued groups are always torsion-free, we know by \cite[Theorem 4.5]{DDMS} that it is sufficient to prove for any $r\geq 1$ that $G_r=\{g\in G\st \omega(g)\geq r\}$ is powerful.  Indeed, given any $x,y\in G_r$, we have that
$$\omega([x,y])\geq \omega(x)+\omega(y)\geq 2r\geq r+1,$$
so the saturation yields some $z\in G_r$ such that $z^p=[x,y]$.  This means that $G_r/(G_r)^p$ is abelian, so $G_r$ is powerful, as desired.
\end{proof}

\subsection{Quasisplit pro-\texorpdfstring{$p$}{p} Iwahori}

We keep the same setup as in Subsection \ref{qsplit-case}, so $p>2$ and $G$ is a quasisplit connected reductive $K$-group for $K/\Q_p$ a finite extension with maximal torus $T$, maximal split torus $S\leq T$, Borel subgroup $B\leq G$, and a fixed absolutely special vertex $x_0\in \mc A(S)$ arising from our Chevalley-Steinberg system $X$.

Suppose that $x$ is a fixed point lying inside a chamber of $\mc B(G)$.  Without loss of generality, we may assume that $x$ lies in the dominant base chamber $C\sub \mc A(S)$ relative to $x_0$ and $B$ in the semisimple apartment associated to $S$, as otherwise we could have changed our choices above to arrange this; write $x=x_0+v$ for $v\in V(S\ad)$.  Then $I\coloneqq G_{x,0+}$ is a maximal open pro-$p$ subgroup of $G(K)^0$ (the kernel of the Kottwitz homomorphism) by \cite[Proposition 13.5.2(1)]{KP}, referred to as a pro-$p$ Iwahori subgroup, and we have defined a saturated $p$-valuation $\omega_x$ on its open subgroup $G_{x,\frac1{p-1}+}$; a natural question to pose is whether or not this subgroup is all of $I$.  Using that $G$ is quasisplit, it is possible to find explicit conditions ensuring that this is the case.

\begin{rem} \label{torus-assumption}
Using the Iwahori decomposition, we see that $G_{x,0+}=G_{x,\frac1{p-1}+}$ if and only if $T(K)_{0+}=T(K)_{\frac1{p-1}+}$ and $U_{\alpha,x,0+}=U_{\alpha,x,\frac1{p-1}+}$ for all $\alpha\in \Phi^{\mathrm{nd}}$.  If $T\iso \prod_{i=1}^m R_{L_i/K}\G_m$ is in fact an induced torus, then the condition that $T(K)_{0+}=T(K)_{\frac1{p-1}+}$ translates exactly to the requirement that $p-1>e_{L_i}$ for all $i$, however for more general weakly induced tori it is difficult to reduce this assumption to a simpler statement.
\end{rem}

On the other hand, we can provide numerical bounds that are equivalent to $U_{\alpha,x,0+}=U_{\alpha,x,\frac1{p-1}+}$.  To set notation, for $\alpha\in \Phi^{+,\mathrm{nd}}$, we will write $m_\alpha\coloneqq 1$ if $\alpha$ is of type R1, and $m_\alpha\coloneqq 2$ if $\alpha$ is of type R2.

\begin{lem} \label{qpslit-zeroplus-roots}
For $\alpha\in \Phi^{+,\mathrm{nd}}$, the condition $U_{\alpha,x,0+}=U_{\alpha,x,\frac1{p-1}+}$ holds if and only if 
$$\frac1{p-1}<\langle \alpha,v\rangle,$$
and $U_{-\alpha,x,0+}=U_{-\alpha,x,\frac1{p-1}+}$ holds if and only if 
$$\langle \alpha,v\rangle<\frac1{m_\alpha}\left(\frac1{e_{K_{\tilde\alpha}}}-\frac1{p-1}\right).$$
\end{lem}

\begin{proof}
Fix any $\alpha\in \Phi^{+,\mathrm{nd}}$, and write $a\coloneqq \langle \alpha,v\rangle$ for convenience.  By \cite[Fact 6.4.5]{KP}, we have that $a\in \left(0,\frac1{m_\alpha e_{K_{\tilde\alpha}}}\right)$.

Our analysis will use similar ideas to those in the proof of Lemma \ref{density}.  Suppose that $\alpha$ is first of type R1, so as in Subsection \ref{R1} we have an isomorphism $u_\alpha:K_{\tilde\alpha}\xto{\sim} U_\alpha(K)$ which satisfies
$$U_{\alpha,x,r+}=\{u_\alpha(x)\st v(x)+a>r\}=u_\alpha(\mf m_{K_{\tilde\alpha}}^i)$$
for all $r\geq 0$, where $i$ is the least integer strictly greater than
$$e_{K_{\tilde\alpha}}\cdot\left(r-a\right),$$
which can be described succinctly as 
$$i=1+\lfloor e_{K_{\tilde\alpha}}\cdot\left(r-a\right)\rfloor.$$
We then see that $U_{\alpha,x,0+}=U_{\alpha,x,\frac1{p-1}+}$ if and only if 
$$\left\lfloor e_{K_{\tilde\alpha}}\cdot \left(\frac1{p-1}-a\right)\right\rfloor=\lfloor e_{K_{\tilde\alpha}}(-a)\rfloor,$$
and the latter is equal to $-1$.  It is easy to see that this condition is thus equivalent to $a>\frac1{p-1}$.  Similarly, $U_{-\alpha,x,0+}=U_{-\alpha,x,\frac1{p-1}+}$ holds if and only if 
$$\left\lfloor e_{K_{\tilde\alpha}}\cdot \left(\frac1{p-1}+a\right)\right\rfloor=\lfloor e_{K_{\tilde\alpha}}\cdot a\rfloor,$$
and the latter automatically equaling 0 means this is equivalent to requiring
$$e_{K_{\tilde\alpha}}\cdot \left(\frac1{p-1}+a\right)<1.$$
This is precisely the same as saying $a<\frac1{e_{K_{\tilde\alpha}}}-\frac1{p-1}$.

Suppose now that $\alpha\in \Phi^{+,\mathrm{nd}}$ is of type R2, so now $0<a<\frac1{2e_{K_{\tilde\alpha}}}$.  Similar to the previous case, we have by Remark \ref{R2-desc} that $U_{\alpha,x,0+}=U_{\alpha,x,\frac1{p-1}+}$ if and only if 
$$\left\lfloor e_{K_{\tilde\alpha}}\cdot\left(\frac1{p-1}-a\right)\right\rfloor=\lfloor e_{K_{\tilde\alpha}}(-a)\rfloor\quad\mathrm{and}\quad\left\lfloor e_{K_{\tilde\alpha}}\cdot\left(\frac1{p-1}-2a\right)\right\rfloor=\lfloor e_{K_{\tilde\alpha}}(-2a)\rfloor.$$
The latter values in both expressions equal $-1$, so this is equivalent to requiring
$$e_{K_{\tilde\alpha}}\cdot\left(\frac1{p-1}-a\right)<0\quad\mathrm{and}\quad e_{K_{\tilde\alpha}}\cdot\left(\frac1{p-1}-2a\right)<0;$$
the former is a stricter condition since $a>0$, and can be rewritten as $a>\frac1{p-1}$.  One also has $U_{-\alpha,x,0+}=U_{-\alpha,x,\frac1{p-1}+}$ if and only if 
$$\left\lfloor e_{K_{\tilde\alpha}}\cdot\left(\frac1{p-1}+a\right)\right\rfloor=\lfloor e_{K_{\tilde\alpha}}\cdot a\rfloor\quad\mathrm{and}\quad\left\lfloor e_{K_{\tilde\alpha}}\cdot\left(\frac1{p-1}+2a\right)\right\rfloor=\lfloor e_{K_{\tilde\alpha}}\cdot 2a\rfloor;$$
in this case, the latter is a stricter condition, and $\lfloor e_{K_{\tilde\alpha}}\cdot 2a\rfloor=0$ means that it can be rewritten as 
$$a<\frac12\left(\frac1{e_{K_{\tilde\alpha}}}-\frac1{p-1}\right).$$
\end{proof}

We obtain the following as a corollary of the argument preceding Lemma \ref{qpslit-zeroplus-roots}.

\begin{cor} \label{polytope}
If $G$ is quasisplit, $T$ is weakly induced, and $T(K)_{0+}=T(K)_{\frac1{p-1}+}$, then there is an open convex polytope $U$ of the dominant base chamber $C\sub \mc A(S)$ given by
$$U=\left\{x\in \mc A(S)\Biggst x=x_0+v, \quad \frac1{p-1}<\langle \alpha, v\rangle<\frac1{m_\alpha}\left(\frac1{e_{K_{\tilde\alpha}}}-\frac1{p-1}\right)\quad \forall \alpha\in \Phi^{+,\mathrm{nd}}\right\}$$
such that each point $x\in U$ corresponds to a distinct saturated $p$-valuation $\omega_x$ on $I$, which is described by the Moy-Prasad filtration at the point $x$.  This subset $U$ is nonempty precisely when 
$$p-1>(m_\alpha \mathrm{ht}(\alpha)+1)\cdot e_{K_{\tilde\alpha}}$$
for all $\alpha\in \Phi^{+,\mathrm{nd}}$.  
\end{cor}

In comparison, we note that by \cite[Fact 6.4.5]{KP}, the dominant base chamber $C$ itself can be described in this manner as
$$C=\left\{x\in \mc A(S)\Biggst x=x_0+v,\quad 0<\langle \alpha, v\rangle <\frac1{m_\alpha e_{K_{\tilde\alpha}}}\quad \forall \alpha\in \Phi^{+,\mathrm{nd}}\right\}.$$

\begin{proof}
One can easily check that $U$ is an open convex polytope using the conditions defining $U$: for each $\alpha\in \Phi^{+,\mathrm{nd}}$, we have a linear functional $V(S\ad)\to \R$ given by $v\mapsto \langle \alpha, v\rangle$.  The set 
$$U_0\coloneqq \left\{v\in V(S\ad)\Biggst \frac1{p-1}<\langle \alpha, v\rangle<\frac1{m_\alpha}\left(\frac1{e_{K_{\tilde\alpha}}}-\frac1{p-1}\right)\quad \forall \alpha\in \Phi^{+,\mathrm{nd}}\right\}$$
is thus the intersection of finitely many open half-spaces, and is bounded as the set of simple roots $\Delta\sub \Phi^{+,\mathrm{nd}}$ forms a basis of the dual space $V(S\ad)^*$.  Then $U=x_0+U_0$ is indeed a convex open polytope in $C$.  

The only claim which remains unproven is the characterization of when $U$ is nonempty.  Say $\Delta=\{\alpha_1,\ldots, \alpha_s\}\sub \Phi^{+,\mathrm{nd}}$ is a labeling for our simple roots, and set $a_i\coloneqq \langle \alpha_i,v\rangle$; choosing our point $x=x_0+v$ for $v\in V(S\ad)$ is equivalent to selecting values $a_i\in \R$ for $i\leq s$.  For each $\alpha\in \Phi^{+,\mathrm{nd}}$, write $\alpha=\sum_i c^\alpha_i\alpha_i$ for $c^\alpha_i\geq 0$ with $\mathrm{ht}(\alpha)=\sum_i c^\alpha_i$.  Now $x\in U$ if and only if
$$\frac1{p-1}<\sum_i c_i^\alpha a_i<\frac1{m_\alpha}\left(\frac1{e_{K_{\tilde\alpha}}}-\frac1{p-1}\right)$$
for all $\alpha\in \Phi^{+,\mathrm{nd}}$.  If this set is nonempty, we claim that we can set all of the $a_i$ to be the same value.  Indeed, choosing $\alpha=\alpha_i$ for each $i$ shows that $a_i>\frac1{p-1}$ for each $i$, and the only remaining conditions required for our choice of $a_i$ to result in a point lying inside $U$ are upper bounds on nonnegative linear combinations of the $a_i$ (as $c^\alpha_i\geq 0$).  Thus, we may shrink our values of $a_i>\frac1{p-1}$ without forcing $x$ to exit the set $U$, proving that there is some $\eps>0$ such that $a_i\coloneqq\frac1{p-1}+\eps$ for all $i$ results in a valid point within $U$.  Set $a\coloneqq \frac1{p-1}+\eps$.  Note that now $\langle \alpha,v\rangle=\sum_i c_i^\alpha a=\mathrm{ht}(\alpha) \cdot a$.

Thus, to characterize when $U$ is nonempty, it is sufficient to check when there is some value $a\in \R$ such that 
$$\frac1{p-1}<\mathrm{ht}(\alpha)\cdot a<\frac1{m_\alpha}\left(\frac1{e_{K_{\tilde\alpha}}}-\frac1{p-1}\right)$$
for all $\alpha\in \Phi^{+,\mathrm{nd}}$, which evidently exists precisely when 
$$p-1>(m_\alpha \mathrm{ht}(\alpha)+1)\cdot e_{K_{\tilde\alpha}}$$
for all $\alpha\in \Phi^{+,\mathrm{nd}}$. 
\end{proof}

\begin{rem} \label{splitcase}
Suppose that $G$ is split and $p-1>e_K h$ for $h=1+\max\{\mathrm{ht}(\alpha)\st \alpha\in \Phi\}$ the Coxeter number.  Under these assumptions, $T$ is certainly weakly induced, and our condition $T(K)_{0+}=T(K)_{\frac1{p-1}+}$ is automatically satisfied by Remark \ref{torus-assumption}.  Moreover, $U$ is nonempty by our above characterization, and one can verify that the barycenter $x$ of $C$ is an element of $U$.  As observed by \cite{dottolehung} (warning: their convention uses $v(p)=e_K$), the $p$-valuation constructed in \cite[Proposition 3.5]{LS} on the pro-$p$ Iwahori subgroup $I=G_{x,0+}$ is precisely the $p$-valuation $\omega_x$ in our notation for $x$ the barycenter of $C$.  Thus, Theorem \ref{general-existence} contains \cite[Proposition 3.5]{LS} as a special case.
\end{rem}

\subsection{Semisimple pro-\texorpdfstring{$p$}{p} Iwahori}

We now consider the case when $G$ is semisimple and quasisplit.

\begin{rem} \label{simplyconn}
Recall that if $G$ is quasisplit, semisimple, and simply connected, then its maximal torus $T$ is induced; more concretely, it can be written as 
$$T\iso \prod_{\alpha\in \Delta} R_{K_{\tilde\alpha}/K}\G_m.$$
Here, the map $R_{K_{\tilde\alpha}/K}\G_m\to T\sc$ is given by the product
$$\alpha^\vee\coloneqq \prod_{\sigma\in \Theta/\Theta_{\tilde\alpha}} \sigma\tilde\alpha,$$
where this notation is justified by the fact that the composition of this with the canonical map $\G_m\to R_{K_{\tilde\alpha}/K} \G_m$ has image lying in $S$, and is precisely $\alpha^\vee:\G_m\to S$.  Thus, if we assume that $p-1>(\mathrm{ht}(\tilde\alpha)+1)\cdot e_{K_{\tilde\alpha}}$, when $G$ is quasisplit, semisimple, and simply connected, it is immediate that $T$ is induced with $T(K)_{0+}=T(K)_{\frac1{p-1}+}$.  

We can use this to better understand the non-simply connected case as follows.  For $G$ a semisimple $K$-group, let $G\sc$ denote its simply connected cover, and let $T\sc$ denote the maximal torus given by the preimage of $T$ under the map $G\sc\to G$.  In this case, the apartment $\mc A(S)=\mc A(S\sc)$ is an affine space over $V(S)=V(S\sc)$, where $S\sc$ denotes the maximal split torus of $G\sc$ lying in the preimage of $S$.
\end{rem}

\begin{cor} \label{qsplit-ss}
Suppose that $G$ is a connected quasisplit semisimple $K$-group, and assume that 
$$p-1>\big(m_\alpha\mathrm{ht}(\alpha)+1\big)\cdot e_{K_{\tilde\alpha}}$$
for all $\alpha\in \Phi^{+,\mathrm{nd}}$.  Then the maximal pro-$p$ subgroups of $T(K)$ and $T\sc(K)$ are canonically isomorphic, and the roots and root groups of $G(K)$ and $G\sc(K)$ are identified as always, allowing us to conclude that the pro-$p$ Iwahori subgroups $I\sc\sub G\sc(K)$ and $I\sub G(K)$ associated to the chamber $C$ are isomorphic.  

In particular, all of Corollary \ref{polytope} applies to such a group $G$: there is an open convex nonempty polytope $U$ of $C\sub \mc A(S)$ given by 
$$U=\left\{x\in \mc A(S)\Biggst x=x_0+v, \quad \frac1{p-1}<\langle \alpha, v\rangle<\frac1{m_\alpha}\left(\frac1{e_{K_{\tilde\alpha}}}-\frac1{p-1}\right)\quad \forall \alpha\in \Phi^{+,\mathrm{nd}}\right\}$$
such that each point $x\in U$ corresponds to a distinct saturated $p$-valuation $\omega_x$ on $I$, which is described by the Moy-Prasad filtration at the point $x$.
\end{cor}

The key difference between Corollary \ref{polytope} and Corollary \ref{qsplit-ss} is that we have dropped the technical condition that $T$ is weakly induced with $T(K)_{0+}=T(K)_{\frac1{p-1}+}$ in the latter, and replaced it with the assumption that $G$ is semisimple.  The idea is that $G$ being semisimple exerts enough control over the torus to allow us to drop our earlier additional conditions.

\begin{proof}
To prove this, we will use an argument from the beginning of the proof of \cite[Proposition 12.1]{totaro}, which was also used in the more recent proof of \cite[Lemma 5.4]{dottolehung}.

Let $Z$ be the kernel of the map $G\sc\to G$, which is a finite subgroup of the center of $G\sc$.  As stated in \cite{totaro}, any prime number $q$ which divides the order of the center of a simply connected semisimple group is at most the maximum of the Coxeter numbers of its simple factors in the algebraic closure.  In particular, this means that $q\leq \max_{\tilde\alpha\in \tilde\Phi} \mathrm{ht}(\tilde\alpha)+1$.  Thus, by our hypotheses, $Z$ has order prime to $p$.

We now have an exact sequence
$$Z(K)\to  G\sc(K)\to G(K)\to H^1(K,Z)$$
by using Galois cohomology, and what we just said shows that $Z(K)$ is a finite abelian group with order prime to $p$, that hence has no pro-$p$ subgroups.  It is a standard fact that $H^1(K,Z)$ is finite (e.g., see \cite[Theorem 6.16]{PR}); we claim that $Z$ having order prime-to-$p$ implies that this finite abelian group $H^1(K,Z)$ also has order prime-to-$p$.  Indeed, we know that $H^1(K,Z)=\dirlim_{L/K\text{ finite}}H^1(\Gal(L/K), Z(L))$, and thus there is some finite $L/K$ for which $H^1(K,Z)=H^1(\Gal(L/K), Z(L))$ as the former is finite.  Now the fact that $Z(L)$ is finite with order prime-to-$p$ gives us immediately that there are no elements of order $p$ in $H^1(K,Z)$.  Thus, our exact sequence $Z(K)\to G\sc(K)\to G(K)\to H^1(K,Z)$ has on its ends finite abelian groups of order prime-to-$p$.

To fix notation, choose $x$ to be a point lying within $U$, so $I=G_{x,0+}$ and $I\sc=(G\sc)_{x,0+}$.  Because $G\sc$ is simply connected, $G\sc(K)^0=G\sc(K)$ (using the notation of \cite[Definition 2.6.23]{KP}) by \cite[Proposition 11.5.4]{KP}, hence $I\sc$ is a maximal pro-$p$ subgroup of $G\sc(K)$ by \cite[Proposition 13.5.2]{KP}.  Because $Z(K)$ has no pro-$p$ subgroups, the map $I\sc\to G(K)$ is injective, and has image a maximal pro-$p$ subgroup of $G(K)$ since $H^1(K,Z)$ has no pro-$p$ subgroups.  
Thus, the natural isogeny $I\sc\to I$ must be an isomorphism as the latter is a pro-$p$ subgroup of $G(K)$.

In order to justify that the induced $p$-valuations on $I$ are coming from the Moy-Prasad filtrations on $G$ and not just from those on $G\sc$, we must show that the isogeny of tori $\phi:T\sc\to T$ arising from $G\sc\to G$ preserves the minimal congruent filtration on the respective tori, i.e., that $\phi(T\sc(K)_r^\mathrm{mc})=T(K)_r^\mathrm{mc}$ for any $r>0$.  Since $T\sc$ is an induced torus, we know by \cite[Corollary B.10.13]{KP} that $T\sc(K)^\mathrm{mc}_r=T\sc(K)^\mathrm{std}_r$ is just the standard filtration.  By the functoriality of the minimal congruent filtration (see \cite[Proposition B.10.10]{KP}), we know that
$$\phi(T\sc(K)_r^\mathrm{std})=\phi(T\sc(K)_r^\mathrm{mc})\sub T(K)_r^\mathrm{mc}.$$
Since the minimal congruent filtration is admissible by \cite[Corollary B.10.13]{KP}, we know that $T(K)_r^\mathrm{mc}\sub T(K)_r^\mathrm{std}$ holds by \cite[Lemma B.10.4]{KP}, so it is sufficient to note that $\phi(T\sc(K)_r^\mathrm{std})\sups T(K)_r^\mathrm{std}$.  To see this claim, note that our isogeny $\phi:T\sc\to T$ induces an injection $X^*(T)\into X^*(T\sc)$ with finite cokernel.  Moreover, because $Z=\ker\phi$ has order prime-to-$p$, we find that the cokernel of $X^*(T)\into X^*(T\sc)$ has order prime-to-$p$; to fix notation, let $n$ be this order.  Suppose that $t\in T(K)_{0+}$ is any element of $T(K)_r^\mathrm{std}$, so by definition $v(\chi(t)-1)\geq r$ for all $\chi\in X^*(T)$.  To simplify notation, we identify $t$ with its unique preimage in $T\sc(K)_{0+}$.  Fix any $\chi_0\in X^*(T\sc)$, so $\chi_0^n$ lies in the image of $X^*(T)$.  By hypothesis, we know $v(\chi_0(t)^n-1)\geq r$.  Since $t\in T\sc(K)_{0+}$, write $\chi_0(t)=1+\varpi_E^i u$ for some finite $E/K$ which splits $T\sc$, where $u\in \O_E^\times$.  Now
$$v(\chi_0(t)^n-1)=v((1+\varpi_E^i u)^n-1)=v(n\varpi_E^iu)=v(\varpi_E^iu)=v(\chi_0(t)-1),$$
so we conclude that in fact $v(\chi_0(t)-1)\geq r$ for all $\chi_0\in X^*(T\sc)$, and hence $t$ lies in $\phi(T\sc(K)_r^\mathrm{std})$.  This proves that $\phi(T\sc(K)_r^\mathrm{mc})=T(K)_r^\mathrm{mc}$ for all $r>0$.
\end{proof}

\begin{note}
We note as a corollary to the proof that, under the hypotheses of Corollary \ref{qsplit-ss}, the pro-$p$ Iwahori $I\leq G(K)$ associated to $C$ is a maximal pro-$p$ subgroup of $G(K)$ (not just of $G(K)^0$, as is provided by \cite[Proposition 13.2.5]{KP}).
\end{note}

\begin{rem}
We note that our results subsume part of \cite[Proposition 12.1]{totaro}.  To show this, we will briefly recall the setup from that proposition.  Let $G$ be a connected semisimple group over $K$.  Then $G$ becomes quasisplit after base changing to $G_E$ over a finite unramified extension $E/K$ by \cite{Steinberg}.  The universal covering of $G_E$ is a product of restrictions of scalars of absolutely simple quasisplit groups $H_L$ over various finite extensions $L/E$.  Suppose that $p-1>h(H_L)e_L$, where $h(H_L)$ is a generalized Coxeter number defined to be
$$h(H_L)\coloneqq\max_{\tilde\alpha\in \tilde\Phi_{H_L}^+}(\mathrm{ht}(\tilde\alpha)+1)\cdot [L_{\tilde\alpha}:L]$$
if $H_L$ has reduced relative root system, and is twice this value if $H_L$ has nonreduced relative root system.  The relevant result of Totaro is that it follows that every pro-$p$ subgroup of $G(K)$ admits a $p$-valuation.

Given the same setup as in \cite{totaro}, we note that Corollary \ref{qsplit-ss} constructs a $p$-valuation on the pro-$p$ Iwahori of $G$ which is moreover saturated. Indeed, by Remark \ref{gen_invariants}, it is sufficient to show that we get a $p$-valuation on a pro-$p$ Iwahori $I$ of $H_L$ for any simple factor $H_L$ that is saturated.  Fix such a simple factor $H_L$.  Indeed, we first note that in Totaro's proof, he mentions that the condition $p-1>h(H_L)e_L$ implies that the maximal torus $T$ of $H_L$ splits over a finite extension $\tilde K/K$ with $e_{\tilde K}<p-1$, which means our condition on the torus in all the results above is automatically satisfied.  Next, we note that if $\tilde\alpha$ is a root of $H_L$ of type R1, then $e_L\cdot [L_{\tilde\alpha}:L]\geq e_{L_{\tilde\alpha}}$ implies automatically that $p-1>(\mathrm{ht}(\tilde\alpha)+1)\cdot e_{L_{\tilde\alpha}}$.  If there is some $\tilde\alpha$ which lies over a root of type R2 in the relative root system, then by the definition of the generalized Coxeter number in Totaro, we find that
$$p-1>2(\mathrm{ht}(\tilde\alpha)+1)\cdot [L_{\tilde\alpha}:L]\cdot e_L>(2\mathrm{ht}(\tilde\alpha)+1)e_{L_{\tilde\alpha}},$$
which implies as claimed that we get a saturated $p$-valuation on $I$.
\end{rem}

\section{The Lie Algebra}

In this section, we seek to generalize the arguments of \cite[Section 3]{dottolehung} to our case and compute the graded Lie algebras associated to the $p$-valuations on the Moy-Prasad filtration subgroups $G_{x,\frac1{p-1}+}$ using a more abstract argument.  This argument also yields a computation of the $\Z_p$-Lie algebra $L_{\omega_x}(G_{x,\frac1{p-1}+})_{\frac1{p-1}+}$, which in turn is a generalization of \cite[Lemma 5.3.1]{BHH+} (which only applied to $G=\GL_2$ and $F/\Q_p$ unramified).  We should also note for posterity that a similar computation to that of \cite[Section 3]{dottolehung} was explicitly performed in \cite[Section 3]{koziol-bharatha-bryan}.

As in Section \ref{moy-prasad}, we will first perform our computations in the quasisplit case, and then proceed via unramified descent.  However, we must handle some generalities before beginning work on any specific cases.


\subsection{Generalities on Lie algebras}


For this subsection, let $K/\Q_p$ be a finite extension and $G$ a linear algebraic $K$-group.

\begin{lem}
Let $X$ be a smooth, separated $K$-scheme of finite type.  Then $X(K)$ has a natural structure as a (locally analytic) $K$-manifold, in the sense of \cite[Section 8]{schneiderbook}, and $\dim_K(X(K))=\dim(X)$.
\end{lem}

\begin{proof}
This is certainly well-known, however the author could not find a precise reference; the proof proceeds exactly as in the case when $X$ is a $\C$-scheme, where this is explained in \cite[Appendix B]{hartshorne} or even \cite{gaga}.  For convenience, we note that the analytic topological structure is built as in \cite{con12} by writing $X=\Spec R$ for $R\iso K[x_1,\ldots, x_n](f_1,\ldots, f_m)$ (since $X$ is of finite type), which induces an embedding $X(K)\into K^n$, from which $X(K)$ is endowed with the subspace topology coming from the usual topology on $K$.  One can show the result is independent of this construction and makes $X(K)$ a manifold over $K$.
\end{proof}

We then get the following immediate corollary.

\begin{cor}
The group $G(K)$ is an analytic $K$-group, in the sense of \cite[Section 13]{schneiderbook}, with $\dim_K(G(K))=\dim(G)$.
\end{cor}

To fix notation, we write $G=\Spec A$, where $A$ is now naturally a Hopf $K$-algebra.  Let $\aug:A\to K$ denote its augmentation map.

Consider now the following two Lie algebras, each endowed with a natural Lie bracket (which, for our purposes, we need not define carefully).  On the algebraic side, following \cite[Section 10b]{milneAG}, we have 
$$\Lie(G)(K)\coloneqq \ker\left(G(K[\eps]/(\eps^2))\xto{\eps\mapsto 0} G(K)\right).$$
By \cite[Proposition 10.28]{milneAG}, we have a $K$-linear isomorphism $\Lie(G)(K)\xto{\sim} \Der_\aug(A,K)$, $\aug+\eps D\mapsto D$, where $\Der_\aug(A,K)$ consists of \textbf{$\aug$-derivations}, i.e., the $K$-linear maps $D:A\to K$ satisfying for all $a,b\in A$ the condition
$$D(ab)=D(a)\aug(b)+\aug(a)D(b).$$

On the analytic side, the $K$-analytic group $G(K)$ has Lie algebra given by
$$\Lie(G(K))=T_e(G(K))$$
the tangent space at the identity, following \cite[Section III.13]{schneiderbook}.  By \cite[Corollary 18.8]{schneiderbook}, we know $G(K)$ is paracompact, so we have by \cite[Proposition 9.18]{schneiderbook} that the $K$-linear map
$$T_e(G(K))\xto{\sim} \Der_e(G(K),K)$$
sending $t\mapsto (\Delta_t)$ is an isomorphism.  Here, $\Der_e(G(K),K)$ is the space of \textbf{$e$-derivations}, i.e., the $K$-linear maps $\Delta:C\an(G(K),K)\to K$ satisfying $\Delta(fg)=\Delta(f)g(e)+f(e)\Delta(g)$ for all $f,g\in C\an(G(K),K)$ (locally $K$-analytic functions $G(K)\to K$); given $t\in T_e(G(K))$, the $e$-derivation $\Delta_t$ is defined by
$$\Delta_t(f)\coloneqq (d_ef)(t)\in K$$
for $f\in C\an(G(K),K)$.

\begin{cons}
We now seek to construct a natural map
$$\Lie(G(K))\to \Lie(G)(K)$$
which will ultimately be an isomorphism of Lie $K$-algebras.  While one could conclude using a shorter argument based on \cite[Theorem 10.23]{milneAG} that these are abstractly isomorphic, we want to have an actual map with which we can perform computations; later, we will be comparing $\O_K$-lattices within these two Lie algebras via the map we construct.

One can summarize our identifications above as
$$\Lie(G(K))\iso \Der_e(G(K),K)\leq \Hom_K(C\an(G(K),K),K)$$
and
$$\Lie(G)(K)\iso \Der_\aug(A,K)\leq \Hom_K(A,K)$$
where $G(K)=\Hom_{K-\alg}(A,K)$.  Note that any element $a\in A$ is a polynomial function on $G(K)$ (merely by evaluation), and in particular is locally analytic; thus, we have a map
$$A\into C\an(G(K),K)$$
which is an inclusion as $A$ is reduced.  (We will often identify $A$ with its image under this map.)  This induces a map 
$$\Hom_K(C\an(G(K),K),K)\to \Hom_K(A,K).$$
\end{cons}

\begin{lem}
This restricts to a map 
$$\Phi_G:\Der_e(G(K),K)\to \Der_\aug(A,K),$$
which is functorial in $G$.
\end{lem}

\begin{proof}
Recall that $e\in G(K)=\Hom_{K-\alg}(A,K)$ is precisely the augmentation map $\aug:A\to K$.  Thus, if we have an $e$-derivation $\Delta$, which means
$$\Delta(fg)=\Delta(f)g(e)+f(e)\Delta(g)$$
for $f,g\in C\an(G(K),K)$, it follows immediately that for $a,b\in A$, 
$$\Delta(ab)=\Delta(a)\aug(b)+\aug(a)\Delta(b).$$
Thus, $\Delta$ restricts to an $\aug$-derivation of $A$.  Functoriality in $G$ is clear.
\end{proof}

\begin{prop} \label{explicit-iso}
The map $\Phi_G$ is an isomorphism of Lie $K$-algebras, providing us with a concrete isomorphism between the analytically and algebraically defined Lie $K$-algebras
$$\Lie(G(K))\iso \Der_e(G(K),K)\xto{\Phi_G} \Der_\aug(A,K)\iso \Lie(G)(K).$$
\end{prop}

\begin{proof}
We first seek to show this when $G=\GL_n$, i.e., that $\Phi_{\GL_n}$ is an isomorphism of Lie algebras over $K$.  Indeed, since $\GL_n(K)\sub M_{n\times n}(K)$ is an open subset, we have a natural identification
$$\Lie(\GL_n(K))=T_e(M_{n\times n}(K))=M_{n\times n}(K);$$
the exercise on \cite[p. 101]{schneiderbook} shows that this identifies the Lie brackets on both sides, where the right side has the Lie bracket coming from its associative algebra structure.  Similarly, \cite[p. 188]{milneAG} shows that a derivation $D\in D_\aug(B,K)$ (for $B=K[x_{ij},\det(x_{ij})\inv]$ the ring of regular functions on $\GL_n$) corresponds bijectively to an element $(a_{ij})\in M_{n\times n}(K)$ by $a_{ij}\coloneqq D(x_{ij})$, and this $K$-linear isomorphism
$$\Lie(G)(K)\iso \Der_\aug(A,K)\iso M_{n\times n}(K)$$
identifies the Lie brackets on both sides.  We then claim that the diagram
\[\begin{tikzcd}[row sep=tiny]
	{\Lie(\GL_n(K))} && {\Lie(\GL_n)(K)} \\
	& {M_{n\times n}(K)}
	\arrow["{\Phi_{\GL_n}}", from=1-1, to=1-3]
	\arrow["\sim", from=1-1, to=2-2]
	\arrow["\sim", from=2-2, to=1-3]
\end{tikzcd}\]
commutes, which would prove that $\Phi_{\GL_n}$ is an isomorphism of $K$-Lie algebras.  Indeed, say $\Delta\in \Der_e(\GL_n(K),K)$, so there exists a unique $M=(m_{ij})\in M_{n\times n}(K)$ such that $\Delta=\Delta_M$, i.e., $\Delta(f)=(d_ef)(M)$ for $f\in C\an(G(K),K)$.  It is now sufficient to show for any $i,j$ that $\Delta(x_{ij})=m_{ij}$, since this would prove that $\Delta$ restricts to the derivation in $\Der_\aug(A,K)$ uniquely determined by $M$.  This follows from the definition on \cite[p. 57]{schneiderbook}: we have the chart
$$\iota: \GL_n(K)\into M_{n\times n}(K),$$
and $\Delta(x_{ij})=(d_e x_{ij})(M)$ is defined by
$$d_ex_{ij}:T_e(\GL_n(K))\xto{\Theta_\iota\inv} M_{n\times n}(K)\xto{D_{I_n}(x_{ij}\circ \iota\inv)}K.$$
This exactly says that $(d_ex_{ij})(M)$ is given by $D_{I_n}(x_{ij}\circ \iota\inv)(M)$, where $D_{I_n}(x_{ij}\circ \iota\inv)$ is the derivative at $I_n$ of $x_{ij}\circ \iota\inv$, which is defined on an open neighborhood of $I_n\in M_{n\times n}(K)$.  But this extends to a linear map $x_{ij}:M_{n\times n}(K)\to K$ defined on the full vector space, hence is its own derivative; this shows that $D_{I_n}(x_{ij}\circ \iota\inv)(M)=m_{ij}$, as desired.  Thus, $\Phi_{\GL_n}$ is an isomorphism of Lie $K$-algebras.

For our more general case, since $G$ is a linear algebraic group, we can pick an embedding $G\into \GL_n$, which induces an embedding $G(K)\into \GL_n(K)$.  We then have a commutative diagram
\[\begin{tikzcd}
	{\Lie(\GL_n(K))} & {\Lie(\GL_n)(K)} \\
	{\Lie(G(K))} & {\Lie(G)(K).}
	\arrow["{\Phi_{\GL_n}}", from=1-1, to=1-2]
	\arrow[hook, from=2-1, to=1-1]
	\arrow["{\Phi_G}", from=2-1, to=2-2]
	\arrow[hook, from=2-2, to=1-2]
\end{tikzcd}\]
where we now know that the top arrow is an isomorphism; this immediately implies that $\Phi_G$ is injective.  This is now a $K$-linear injective map of $K$-vector spaces of equal dimension (both equal to $\dim_K(G(K))=\dim(G)$ by \cite[p. 101]{schneiderbook} and \cite[Proposition 1.37]{milneAG}), hence is a $K$-linear isomorphism.  But we know the vertical arrows preserve Lie brackets by naturality, and the top arrow preserves Lie brackets by our earlier argument, so we can conclude that $\Phi_G$ is an isomorphism of Lie $K$-algebras.
\end{proof}

\begin{rem} \label{int-model-rem}
Let $\mc G=\Spec B$ be an integral model of $G$ over $\O_K$, in the sense of \cite[Definition 2.10.1]{KP}, which we assume to be a linear algebraic group over $\O_K$ in a way compatible with the structure on $G$; thus, we may also use $\aug$ to denote the augmentation map $B\to \O_K$.  By \cite[Remark 2.10.8]{KP}, the Lie algebra 
$$\Lie(\mc G)(\O_K)\coloneqq \ker\left(\mc G(\O_K[\eps]/(\eps^2)\xto{\eps\mapsto 0} \mc G(\O_K)\right)$$
is canonically identified with $\Der_\aug(B,\O_K)$ via $\aug+\eps D\mapsto D$, where $\Der_\aug(B,\O_K)$ is defined to be the $\aug$-derivations, i.e., the $\O_K$-linear maps $D$ satisfying 
$$D(ab)=D(a)\aug(b)+\aug(a)D(b)$$
for $a,b\in B$.  Then the same remark says that the canonical map $\Lie(\mc G)(\O_K)\into \Lie(G)(K)$ is injective with image an $\O_K$-lattice in this $K$-vector space.

We also note that \cite[Remark 2.10.3]{KP} says $\mc G(\O_K)$ is open in $G(K)$, hence $\mc G(\O_K)$ is also an analytic $K$-group.
\end{rem}

\subsection{The quasisplit Lie algebra} \label{qsplit-LieAlg}

In this subsection, we take the same setup as Subsection \ref{qsplit-case}, with the additional assumption that $T$ is weakly induced.  That is, $p>2$, $G$ is a quasisplit connected reductive $K$-group, $x\in \mc B(G)$ is any fixed point, $T\leq G$ is a maximally $K$-split maximal $K$-torus assumed to be weakly induced, $S\leq T$ its maximal $K$-split subtorus, $B\leq G$ is a Borel subgroup containing $T$, $\mc A(S)$ is the apartment associated to $S$, $x_0\in \mc A(S)$ is an absolutely special vertex arising from a fixed Chevalley-Steinberg system $X$, all chosen such that $x$ lies in $\mc A(S)$.  We write $x=x_0+v$ for $v\in V(S\ad)$.

\begin{rem}
Following \cite[Section 8.5]{KP} (which still applies because $G$ is quasisplit despite $K$ having proper unramified extensions; see \cite[13.3.3]{KP}, for example), there exist smooth integral $\O_K$-models $\mc G_{x,r}$ and $\mc T_r$ of $G$ and $T$, as well as $\mc U_{\alpha,x,r}$ of $U_\alpha$ for $\alpha\in \Phi^{+,\mathrm{nd}}$ for all $r\geq 0$, satisfying the following:
$$G_{x,r}=\mc G_{x,r}(\O_K),\qquad T(K)_r=\mc T_r(\O_K),\qquad U_{\alpha, x, r}=\mc U_{\alpha, x, r}(\O_K).$$
Write $\mf g_{x,r}=\Lie(\mc G_{x,r})(\O_K)$, $\mf t_r=\Lie(\mc T_r)(\O_K)$, and $\mf u_{\alpha, x, r}=\Lie(\mc U_{\alpha, x, r})(\O_K)$, all of which are $\O_K$-lattices in their respective ambient Lie $K$-algebras $\Lie(G)(K)$, $\Lie(T)(K)$, and $\Lie(U_\alpha)(K)$.  By \cite[8.5.20]{KP}, we have for any $r\geq 0$ a natural identification of $\O_K$-modules
\begin{equation}\label{OK-lattices}
\mf g_{x,r}=\bigoplus_{\alpha\in \Phi^{-,\mathrm{nd}}} \mf u_{\alpha, x, r}\oplus \mf t_r \oplus\bigoplus_{\alpha\in \Phi^{+,\mathrm{nd}}} \mf u_{\alpha, x, r}
\end{equation}
lying within
\begin{equation}\label{K-vs}
\Lie(G)(K)=\bigoplus_{\alpha\in \Phi^{-,\mathrm{nd}}} \Lie(U_\alpha)(K)\oplus \Lie(T)(K)\oplus\bigoplus_{\alpha\in \Phi^{+,\mathrm{nd}}}\Lie(U_\alpha).
\end{equation}
\end{rem}

\begin{rem}
Recall by Proposition \ref{existence} that we have a saturated $p$-valuation $\omega\coloneqq\omega_x$ on $G_{x,\frac1{p-1}+}$ given by
$$\omega(g)=\sup\{r>0\st g\in G_{x,r}\}.$$
Moreover, we have for each $r>\frac1{p-1}$ an Iwahori decomposition
$$\prod_{\alpha\in \Phi^{-,\mathrm{nd}}} U_{\alpha, x, r}\times T(K)_r\times \prod_{\alpha\in \Phi^{+,\mathrm{nd}}} U_{\alpha, x, r}\xto{\sim} G_{x,r}$$
and $\omega|_{G_{x,r}}$ is compatible with this decomposition for all $r>\frac1{p-1}$.  One can easily show using Remark \ref{compatible-saturation} and \cite[Proposition 31.2]{schneiderbook}, which gives us a way of creating explicit $\Q_p$-bases of all such Lie algebras, that there is a natural identification of $\Q_p$-vector spaces
\begin{equation}\label{Qp-vs-0}
L_\omega(G_{x,\frac1{p-1}+})=\bigoplus_{\alpha\in \Phi^{-,\mathrm{nd}}} L_\omega(U_{\alpha,x, \frac1{p-1}+})\oplus L_\omega(T(K)_{\frac1{p-1}+})\oplus\bigoplus_{\alpha\in \Phi^{+,\mathrm{nd}}} L_\omega(U_{\alpha, x, \frac1{p-1}+}).
\end{equation}
One can further use the explicit description of the filtrations $\tilde\omega$ on each of these Lie $\Q_p$-algebras from \cite[Proposition 31.2]{schneiderbook} to conclude that, within this decomposition, there is an identification of $\Z_p$-lattices
\begin{equation}\label{Zp-lattices}
L_\omega(G_{x,\frac1{p-1}+})_r=\bigoplus_{\alpha\in \Phi^{-,\mathrm{nd}}} L_\omega(U_{\alpha,x, \frac1{p-1}+})_r\oplus L_\omega(T(K)_{\frac1{p-1}+})_r\oplus\bigoplus_{\alpha\in \Phi^{+,\mathrm{nd}}} L_\omega(U_{\alpha, x, \frac1{p-1}+})_r
\end{equation}
for every $r>\frac1{p-1}$.  We note for future use that $L_\omega(G_{x,\frac1{p-1}+})_r=\log(G_{x,r})$ for all $r>\frac1{p-1}$ by \cite[p. 220]{schneiderbook}, and the same holds for each of the terms in our Iwahori decomposition \ref{Zp-lattices}.  We also remark that one can easily show that there is a natural identification of $\Q_p$-vector spaces
\begin{equation}\label{Qp-vs}
\Lie(G(K))=\bigoplus_{\alpha\in \Phi^{-,\mathrm{nd}}} \Lie(U_\alpha(K))\oplus\Lie(T(K))\oplus \bigoplus_{\alpha\in \Phi^{+,\mathrm{nd}}} \Lie(U_\alpha(K)).
\end{equation}
\end{rem}

We now seek to prove the following.

\begin{prop} \label{qsplit-LieAlg-iso}
There is a natural isomorphism of Lie $\Z_p$-algebras
$$L_\omega(G_{x,\frac1{p-1}+})_{\frac1{p-1}+}\iso \mf g_{x,\frac1{p-1}+}$$
which respects the filtrations; that is, it induces for all $r>\frac1{p-1}$ an isomorphism
$$L_\omega(G_{x,\frac1{p-1}+})_r\iso \mf g_{x,r}.$$
\end{prop}

The strategy will be as follows.  Fix any $r>\frac1{p-1}$.  The Lie $\Z_p$-algebra $L_\omega(G_{x,\frac1{p-1}+})_r=\log(G_{x,r})$ is a lattice within the Lie $\Q_p$-algebra $L_\omega(G_{x,\frac1{p-1}+})$, which is naturally isomorphic to $\Lie(G_{x,\frac1{p-1}+})$ by \cite[p. 231]{schneiderbook}.  Since $G_{x,\frac1{p-1}+}$ is an open subgroup of $G(K)$, then $\Lie(G_{x,\frac1{p-1}+})=\Lie(G(K))$, and we have an explicit isomorphism of Lie $K$-algebras
$$\Lie(G(K))\iso \Lie(G)(K)$$
by Proposition \ref{explicit-iso}.  Naturally lying within the right side is the $\O_K$-lattice 
$$\mf g_{x,r}=\Lie(\mc G_{x,r})(\O_K)$$
by Remark \ref{int-model-rem}.  Altogether, this information can be arranged into the following diagram:
\[\begin{tikzcd}[cramped, row sep=tiny]
	{\log(G_{x,r})} & {L_\omega(G_{x,\frac1{p-1}+})_r} & {L_\omega(G_{x,\frac1{p-1}+})} & {\Lie(G(K))} & \\
	&&&& {\Lie(G)(K)}. \\
	&&{\mf g_{x,r}} & {\Lie(\mc G_{x,r})(\O_K)}
	\arrow[between={0.2}{0.8}, equals, nfold, from=1-1, to=1-2]
	\arrow[hook, from=1-2, to=1-3]
	\arrow[between={0.2}{0.8}, equals, nfold, from=1-3, to=1-4]
	\arrow["\sim", from=1-4, to=2-5]
	\arrow[between={0.2}{0.8}, equals, nfold, from=3-3, to=3-4]
	\arrow[hook, from=3-4, to=2-5]
\end{tikzcd}\]
In order to prove Proposition \ref{qsplit-LieAlg-iso}, because every map preserves Lie brackets in this diagram, it is sufficient to show that the image of $\log(G_{x,r})$ under the top row is equal to the image of $\mf g_{x,r}$ under the bottom row.

One can readily verify that the maps in this diagram are all compatible with our Iwahori decompositions above (\ref{OK-lattices}, \ref{K-vs}, \ref{Qp-vs-0}, \ref{Zp-lattices}, \ref{Qp-vs}).  Thus, as we need only show that our two lattices have equal image, it is sufficient to prove the following two claims.

\begin{lem} \label{LieAlg-cases}
Let $r>\frac1{p-1}$ be fixed.
\begin{enumerate}
\item \label{LieAlg-cases-torus} The images of $\log(T(K)_r)$ and $\mf t_r$ under these two maps coincide.
\item \label{LieAlg-cases-root} For each $\alpha\in \Phi^{\mathrm{nd}}$, the images of $\log(U_{\alpha,x,r})$ and $\mf u_{\alpha, x, r}$ under these two maps coincide.
\end{enumerate}
\end{lem}

We will begin with proving part \ref{LieAlg-cases-root} in the case when $\alpha$ is a root of type R1; as a corollary of the proof, we obtain part \ref{LieAlg-cases-torus} when $T$ is a split torus.  We will then generalize to the case of any weakly induced torus via a descent argument.  Finally, we will prove part \ref{LieAlg-cases-root} when $\alpha$ is a root of type R2.

\begin{claim} \label{R1-LieAlg}
Part \ref{LieAlg-cases-root} of Lemma \ref{LieAlg-cases} is true when $\alpha$ is of type R1.
\end{claim}

\begin{proof}
Fix $r>\frac1{p-1}$, and let $L=K_{\tilde\alpha}$ for notational simplicity.  By \cite[Construction 2.9.16]{KP}, there is an isomorphism
$$\WR_{L/K}\G_a\xto{\sim} U_\alpha,$$
which induces an isomorphism $u_\alpha: L\xto{\sim} U_\alpha(K)$.  Now 
$$U_{\alpha, x, r}=\{u_\alpha(x)\st x\in \mf m_L^{n_r}\},$$
where $n_r$ is the least integer $\geq e_L\cdot (r-\langle \alpha,v\rangle)$.  This has associated integral model given by $\WR_{\mc O_L/\mc O_K} \G_{a,r}$, where we write
$$\G_{a,r}=\Spec \O_L[\varpi_L^{-n_r}t]$$
as in \cite[C.2]{KP}.  We also set notation and write $\G_a=\Spec L[t]$.  Notice that 
$$\Lie(\WR_{L/K} \G_a)(K)=\Lie(\G_a)(L)\quad \mathrm{and}\quad \Lie(\WR_{\O_L/\O_K}\G_{a,r})(\O_K)=\Lie(\G_{a,r})(\O_L);$$ thus, to prove our claim, we may suppose without loss of generality that $L=K$ (as proving this for an arbitrary finite extension $K/\Q_p$ will give us our case, with fewer notational burdens).  We can then work directly with $\G_a$ and its integral model $\G_{a,r}$ via the isomorphisms above.

We now have a $K$-linear isomorphism
$$\Der_\aug(K[t],K)\xto{\sim} K,\qquad D\mapsto D(t)$$
as well as an $\O_K$-linear isomorphism
$$\Der_\aug(\O_K[\varpi_K^{-n_r}t],\O_K)\xto{\sim} \mf m_K^{n_r}=\{x\in K\st v(x)\geq r-\langle \alpha,v\rangle\},\qquad D\mapsto D(t).$$
In fact, one can easily verify that these fit into a commutative diagram
\[\begin{tikzcd}[cramped]
	{\Lie(\G_{a,r})(\O_K)} & {\Der_\aug(\O_K[\varpi_K^{-n_r}t],\O_K)} & {\mf m_K^{n_r}} \\
	{\Lie(\G_a)(K)} & {\Der_\aug(K[t],K)} & K.
	\arrow["\sim", from=1-1, to=1-2]
	\arrow[hook, from=1-1, to=2-1]
	\arrow["\sim", from=1-2, to=1-3]
	\arrow[hook, from=1-3, to=2-3]
	\arrow["\sim", from=2-1, to=2-2]
	\arrow["\sim", from=2-2, to=2-3]
\end{tikzcd}\]

Next, we compute the analytic side.  Via $g\mapsto g(t)$, we have identifications fitting into a commutative diagram
\[\begin{tikzcd}[cramped]
	{\G_{a,r}(\O_K)} & {m_K^{n_r}} \\
	{\G_a(K)} & {K.}
	\arrow["\sim", from=1-1, to=1-2]
	\arrow[hook, from=1-1, to=2-1]
	\arrow[hook, from=1-2, to=2-2]
	\arrow["\sim", from=2-1, to=2-2]
\end{tikzcd}\]
The inclusion $\iota:\mf m_K^{n_r}\into K$ is itself a chart for our $p$-adic analytic group $\mf m_K^{n_r}$, giving us an identification $\Lie(\mf m_K^{n_r})\xto{\sim} K$.  We claim then that the composition
\begin{equation}\label{iota-comp}
\mf m_K^{n_r}\xto{\log} \log(\mf m_K^{n_r})\into \Lie(\mf m_K^{n_r})\xto{\sim} K\end{equation}
is merely $\iota$, where $\log$ is defined as in \cite[Section 31]{schneiderbook} via our $p$-valuation.  Indeed, by choosing an ordered basis $(b_i)_{i\leq \ell}$ of $\O_K$ (see Remark \ref{OKbasis}), one can construct an explicit isomorphism of $\Q_p$-vector spaces $L_\omega(\mf m_K^{n_r})\xto{\sim} K$ by sending our $\Q_p$-basis vectors $\log(\varpi_K^{n_r}b_i)$ each to $\varpi_K^{n_r} b_i$.  In particular, this allows us to put a $K$-linear structure on $L_\omega(\mf m_K^{n_r})$, after which
$$\log:\mf m_K^{n_r}\to \log(\mf m_K^{n_r})\into L_\omega(\mf m_K^{n_r})$$
is in fact a $K$-analytic chart which induces a map $\Lie(\mf m_K^{n_r})\xto{\sim} L_\omega(\mf m_K^{n_r})$; a priori, it was only known to be $\Q_p$-analytic.  It now follows that the diagram
\[\begin{tikzcd}[row sep=tiny]
	{L_\omega(\mf m_K^{n_r})} & \\
	& {\Lie(\mf m_K^{n_r})} \\
	K
	\arrow["\sim"', from=2-2, to=1-1]
	\arrow["\sim", from=1-1, to=3-1]
	\arrow["\sim"', from=2-2, to=3-1]
\end{tikzcd}\]
commutes, from which our desired fact that \ref{iota-comp} is the inclusion $\iota$ follows.

We have now proven that both
$$\log(\mf m_K^{n_r})\into \Lie(\mf m_K^{n_r})=\Lie(K)=K$$
and
$$\Lie(\G_{a,r})(\O_K)\into \Lie(\G_a)(K)\xto{\sim} \Der_\aug (K[t],K)\xto{\sim} K$$
have image equal to $\mf m_K^{n_r}$.  The only part which remains is to verify that
\begin{equation}\label{Ga-diag}\begin{tikzcd}[row sep=tiny]
	{\Lie(K)} & \\
	& K \\
	{\Lie(\G_a)(K)}
	\arrow["\sim", from=1-1, to=2-2]
	\arrow["{\Phi_{\G_a}}"', from=1-1, to=3-1]
	\arrow["\sim"', from=3-1, to=2-2]
\end{tikzcd}\end{equation}
commutes.  This is easily verified: given any $a\in K$, we have a corresponding $\Delta_a\in \Der_e(K,K)$ sending $f\in C\an(K,K)$ to $\Delta_af\coloneqq (d_0 f)(a)$.  Then if $f\in K[t]$ (being viewed as a polynomial on $K$), we have that $(d_0f)(a)=f_1a$ if $f=\sum f_it^i$, so $\Delta_a|_{K[t]}\in \Der_\aug(K[t],K)$ has image under $D\mapsto D(t)$ given precisely by $\Delta_a(t)=a$.  This proves that \ref{Ga-diag} commutes, giving us our claim.
\end{proof}

\begin{claim} \label{splitTor-LieAlg}
Part \ref{LieAlg-cases-torus} of Lemma \ref{LieAlg-cases} is true when $T$ is a split $K$-torus.
\end{claim}


\begin{proof}
Fix $r>\frac1{p-1}$.  Since $T$ is split, it is sufficient to prove this when $T=\G_m=\Spec K[t,t\inv]$.

First, we compute the algebraic side.  In this case, by \cite[Example A.5.21]{KP}, our $\O_K$-integral model is given by
$$\G_{m,r}=\Spec\O_K[\varpi_K^{-n_r}(t-1),t\inv]$$
where $n_r=\lceil e_Kr\rceil$ is the least integer such that $\frac{n_r}{e_K}\geq r$.  Now we have a $K$-linear isomorphism
$$\Der_\aug(K[t,t\inv],K)\xto{\sim} K,\qquad D\mapsto D(t),$$
as well as an $\O_K$-linear isomorphism
$$\Der_\aug(\O_K[\varpi_K^{-n_r}(t-1),t\inv],\O_K)\xto{\sim} \mf m_K^{n_r}=\{x\in K\st v(x)\geq r\},\qquad D\mapsto D(t).$$
Once again, one can easily check that the diagram
\[\begin{tikzcd}[cramped]
	{\Lie(\G_{m,r})(\O_K)} & {\Der_\aug(\O_K[\varpi_K^{-n_r}(t-1),t\inv],\O_K)} & {\mf m_K^{n_r}} \\
	{\Lie(\G_m)(K)} & {\Der_\aug(K[t,t\inv],K)} & K
	\arrow["\sim", from=1-1, to=1-2]
	\arrow[hook, from=1-1, to=2-1]
	\arrow["\sim", from=1-2, to=1-3]
	\arrow[hook, from=1-3, to=2-3]
	\arrow["\sim", from=2-1, to=2-2]
	\arrow["\sim", from=2-2, to=2-3]
\end{tikzcd}\]
commutes.

Next, we compute the analytic side.  Via $g\mapsto g(t)$, we have identifications fitting into a commutative diagram
\[\begin{tikzcd}[cramped]
	{\G_{m,r}(\O_K)} & {1+\mf m_K^{n_r}} \\
	{\G_m(K)} & {K^\times.}
	\arrow["\sim", from=1-1, to=1-2]
	\arrow[hook, from=1-1, to=2-1]
	\arrow[hook, from=1-2, to=2-2]
	\arrow["\sim", from=2-1, to=2-2]
\end{tikzcd}\]
We have an identification $\Lie(1+\mf m_K^{n_r})\xto{\sim} K$ coming from the inclusion $1+\mf m_K^{n_r}\into K$, and we claim that the composition
$$1+\mf m_K^{n_r}\xto{\log} \log(1+\mf m_K^{n_r})\into \Lie(1+\mf m_K^{n_r})\xto{\sim} K$$
has image precisely $\mf m_K^{n_r}$.  In fact, this follows from the analogous claim in the proof of Claim \ref{R1-LieAlg} (see \ref{iota-comp}): the usual logarithm (defined in, e.g., \cite[Proposition 5.5]{neukirch}, which exists as $n_r\geq e_Kr>\frac{e_K}{p-1}$) is an isomorphism of $p$-valued groups $1+\mf m_K^{n_r}\xto{\log} \mf m_K^{n_r}$ (see Remark \ref{Gm-log-exp}).  We also have an identification $\Lie(\mf m_K^{n_r})\xto{\sim} K$ coming from the inclusion $\mf m_K^{n_r}\into K$, and from the power series of $\log$, one can verify that the map it induces on Lie algebras after these identifications is precisely the identity $K\to K$.  Thus, our claim literally follows directly from the analogous claim we previously proved.

Just as in the proof of Claim \ref{R1-LieAlg}, all that remains to be proven is that
\begin{equation}\label{Gm-diag}\begin{tikzcd}[row sep=tiny]
	{\Lie(K^\times)} & \\
	& K \\
	{\Lie(\G_m)(K)}
	\arrow["\sim", from=1-1, to=2-2]
	\arrow["{\Phi_{\G_m}}"', from=1-1, to=3-1]
	\arrow["\sim"', from=3-1, to=2-2]
\end{tikzcd}\end{equation}
commutes.  Indeed, given any $a\in K$, we have a corresponding $\Delta_a\in \Der_e(K^\times, K)$ sending $f\in C\an(K^\times, K)$ to $\Delta_af\coloneqq (d_1f)(a)$.  Then viewing $t\in K[t,t\inv]$ as a polynomial on $K^\times$, we find that $\Delta_a(t)=a$ (as $d_1t$ is just the identity $K\to K$).  This proves our claim.
\end{proof}

\begin{rem} \label{cochars}
We note the following as a corollary of the proof above.  Say $r>\frac1{p-1}$.  If $T$ is a split torus, we have by \cite[Lemma B.2.2]{KP} an identification
$$T(K)_r=X_*(T)\otimes_\Z (1+\mf m_K^{n_r}),$$
where $n_r=\lceil e_K r\rceil$.  We similar have identifications
$$\Lie(T)(K)=X_*(T)\otimes_\Z K\qquad\mathrm{and}\qquad\Lie(\mc T_r)(\O_K)=X_*(T)\otimes_\Z \mf m_K^{n_r}$$
by \cite[B.6.8 and 2.5.22]{KP}.  After all of these identifications, the proof of Claim \ref{splitTor-LieAlg} above shows that the map
$$T(K)_r\xto{\log} \log(T(K)_r)\to \Lie(T(K))\to \Lie(T)(K)$$
is merely the map $X_*(T)\otimes_\Z (1+\mf m_K^{n_r})\to X_*(T)\otimes_\Z K$ induced by the usual logarithm $\log:1+\mf m_K^{n_r}\xto{\sim} \mf m_K^{n_r}\into K$.  Similarly, the map $\Lie(\mc T_r)(\O_K)\to \Lie(T)(K)$ viewed as $X_*(T)\otimes_\Z \mf m_K^{n_r}\to X_*(T)\otimes_\Z K$ is merely induced by the inclusion $\mf m_K^{n_r}\into K$.
\end{rem}

\begin{claim} \label{genTor-LieAlg}
Part \ref{LieAlg-cases-torus} of Lemma \ref{LieAlg-cases} is true when $T$ is a weakly induced $K$-torus.
\end{claim}

\begin{proof}
Because $T$ is weakly induced, the minimal congruent filtration is the standard filtration by \cite[Corollary B.10.13]{KP}.  Fix now $r>\frac1{p-1}$, and let $L/K$ be a finite Galois extension for which $T_L$ is a split torus.  By Remark \ref{cochars} above, we know that we can interpret Claim \ref{splitTor-LieAlg} as saying that the images of the two maps
\[\begin{tikzcd}[row sep=tiny, cramped]
	{X_*(T)\otimes_\Z (1+\mf m_L^{n_re(L/K)})} & \\
	& {X_*(T)\otimes_\Z L} \\
	{X_*(T)\otimes_\Z \mf m_L^{n_r e(L/K)}}
	\arrow[from=1-1, to=2-2]
	\arrow[from=3-1, to=2-2]
\end{tikzcd}\]
are equal; here, the top one is induced by the usual logarithm, and the bottom is induced by the inclusion.  Note that these maps are both evidently $\Gal(L/K)$-equivariant, so we can take invariants to conclude that the images of the two maps
\[\begin{tikzcd}[row sep=tiny, cramped]
	{\left(X_*(T)\otimes_\Z (1+\mf m_L^{n_re(L/K)})\right)^{\Gal(L/K)}} & \\
	& {\left(X_*(T)\otimes_\Z L\right)^{\Gal(L/K)}} \\
	{\left(X_*(T)\otimes_\Z \mf m_L^{n_r e(L/K)}\right)^{\Gal(L/K)}}
	\arrow[from=1-1, to=2-2]
	\arrow[from=3-1, to=2-2]
\end{tikzcd}\]
are equal.  By \cite[B.6.8 and 2.5.22]{KP}, this diagram can be rewritten as
\[\begin{tikzcd}[row sep=tiny, cramped]
	{T(K)_r} & \\
	& {\Lie(T)(K)} \\
	{\Lie(\mc T_r)(\O_K)}
	\arrow[from=1-1, to=2-2]
	\arrow[from=3-1, to=2-2]
\end{tikzcd}\]
which contains precisely the two maps which we needed to show have equal image to complete the proof of this claim.
\end{proof}

\begin{claim} \label{R2-LieAlg}
Part \ref{LieAlg-cases-root} of Lemma \ref{LieAlg-cases} is true when $\alpha$ is of type R2.
\end{claim}

\begin{proof}
Fix $r>\frac1{p-1}$, and let $\alpha\in \Phi^\mathrm{nd}$ be of type R2.  Keeping the same setup as in the proof of Claim \ref{R1-LieAlg}, we get an isomorphism of $K$-groups
\begin{equation} \label{R2-isom-WR}\WR_{K_{\tilde\beta}/K} U_{K_{\tilde\alpha}/K_{\tilde\beta}}\xto{\sim} U_\alpha\end{equation}
induced by our Chevalley-Steinberg system $X$.  Exactly as in that proof, we may ignore the Weil restriction without loss of generality, as can be justified in general by \cite[Lemma A.3.12]{KP}.  Thus, we will simplify notation and replace $K_{\tilde\beta}$ by $K$, and $K_{\tilde\alpha}$ by $L$.  Recall that
$$U_{L/K}(R)=\{(s,t)\in (R\otimes_L K)^2\st t+\ol t=s\ol s\}$$
for any $K$-algebra $R$.  Let $U_{L/K,r}$ denote the subgroup of $U_{L/K}(K)$ corresponding to $U_{\alpha, x,r}$ via the isomorphism \ref{R2-isom-WR}, so
$$U_{L/K,r}=\left\{(s,t)\in U_{L/K}(K)\st v(s)\geq r-a,\quad v(t)\geq r-2a\right\}$$
where $a=\langle \alpha, v\rangle$.

In order to compute the algebraic side, we need to have an explicit description of the appropriate integral model of $U_{L/K}$ associated to $U_{L/K,r}$.  One description is given in \cite[Section C.4]{KP}, however we will construct a new model using an argument inspired by their construction which more closely matches Proposition \ref{compatibility}.  This will agree with their construction by \cite[Corollary 2.10.11]{KP}.  As in Proposition \ref{compatibility}, let $\gamma\in \O_L$ be a traceless element such that $\{1,\gamma\}$ is an $\O_K$-basis of $\O_L$.  Then we have an isomorphism of pointed $K$-schemes 
\begin{align*}
\G_a\times\G_a\times\G_a\to & \, U_{L/K},\\
(x,y,z)\mapsto& \left(x,\frac12x^2\right)\cdot \left(y\gamma,-\frac12y^2\gamma^2\right)\cdot (0,z\gamma)\\
&=\left(x+y\gamma, \frac12(x^2-y^2\gamma^2)+\gamma(z-xy)\right).
\end{align*}
Indeed, one can verify that the inverse of this is given by
$$(u,v)\mapsto \left(\frac12(u+\ol u),\quad \frac1{2\gamma}(u-\ol u),\quad \frac1{2\gamma}(v-\ol v)+\frac1{4\gamma}(u^2-\ol u^2)\right).$$
We can transport the group structure from $U_{L/K}$ to one on $(\G_a)^3$; after doing this, the restriction to each factor $\G_a$ recovers the usual group structure on $\G_a$.  Moreover, the proof of Proposition \ref{compatibility} shows that $U_{L/K,r}$ is precisely the image of 
$$\G_{a,r_1}(\O_K)\times \G_{a,r_2}(\O_K)\times \G_{a,r_3}(\O_K)$$
for $r_1=r-a$, $r_2=r-a-v(\gamma)$, and $r_3=r-2a-v(\gamma)$ (using the convention that $\G_{a,s}$ is the standard integral model for $\G_a$ satisfying $\G_{a,s}(\O_K)=\{x\in K\st v(x)\geq s\}$).  Just as in our original reduction to only needing to consider Iwahori factors in order to prove Proposition \ref{qsplit-LieAlg-iso}, in order to prove this claim, we can reduce to the case of proving it individually for each of these copies of $\G_a$ with their respective filtrations and integral models.  However, this was already accomplished in the proof of Claim \ref{R1-LieAlg}, so three applications of that result gives us our desired fact in this claim.
\end{proof}

This completes the proof of Proposition \ref{qsplit-LieAlg-iso}, giving us our desired computation of the Lie $\Z_p$-algebra $L_\omega(G_{x,\frac1{p-1}+})_{\frac1{p-1}+}$ when $G$ is quasisplit.

\begin{rem} \label{qsplit-expl-desc}
Fix $r>\frac1{p-1}$ and fix a Chevalley-Steinberg system, from which there arises an absolutely special point $x_0\in \mc A(S)$, and we can write $x=x_0+v$ for $v\in V(S\ad)$ (as in Remark \ref{root-filts}).  To ease notation, write $\Theta=\Gal(\ol K/K)$.  Given $\alpha\in \Phi^\mathrm{nd}$, we can describe the isomorphism of Lie $\Z_p$-algebras
\begin{equation}\label{root-LieAlg}
\log(U_{\alpha, x, r})\xto{\sim} \mf u_{\alpha, x, r}\end{equation}
by chasing through the proofs of Claims \ref{R1-LieAlg} and \ref{R2-LieAlg} as follows.
\begin{itemize}
\item If $\alpha$ is of type R1, then \ref{root-LieAlg} sends 
$$\log(u_\alpha(x))\mapsto \sum_{\sigma\in \Theta/\Theta_{\tilde\alpha}} \sigma(x\cdot X_{\tilde\alpha})$$
for $x\in K_{\tilde\alpha}$ with $v(x)\geq r-\langle \alpha,v\rangle$.
\item If $\alpha$ is of type R2, let $\gamma\in \O_{K_{\tilde\alpha}}$ be a traceless element such that $\{1,\gamma\}$ is an $\O_{K_{\tilde\beta}}$-basis of $\O_{K_{\tilde\alpha}}$, which exists by Remark \ref{tracelessexistence}.  Then \ref{root-LieAlg} sends
\begin{align*}
\log\left(u_\alpha\left(x,\frac12x^2\right)\right)&\mapsto \sum_{\sigma\in \Theta/\Theta_{\tilde\beta}} \sigma(x\cdot X_{\tilde\alpha}),\\
\log\left(u_\alpha\left(y\gamma,-\frac12y^2\gamma^2\right)\right)&\mapsto \sum_{\sigma\in \Theta/\Theta_{\tilde\beta}} \sigma(y\gamma\cdot X_{\tilde\alpha}),\\
\log\left(u_\alpha\left(0,z\gamma\right)\right)&\mapsto \sum_{\sigma\in \Theta/\Theta_{\tilde\beta}} \sigma(z\gamma\cdot X_{\tilde\beta}),
\end{align*}
where $x,y,z\in K_{\tilde\beta}$ satisfy
$$v(x)\geq r-\langle \alpha,v\rangle,\quad v(y)\geq r-\langle \alpha, v\rangle-v(\gamma),\quad v(z)\geq r-2\langle \alpha, v\rangle -v(\gamma).$$
\end{itemize}
\end{rem}

\subsection{The Lie algebra for more general reductive groups}

We now take the same setup as in Subsection \ref{general-case}.  That is, $p>2$, $G$ is any connected reductive $K$-group for $K/\Q_p$ a finite extension, and we assume that $G$ has a weakly induced maximally $K$-split maximal $K$-torus.  Fix any $x\in \mc B(G)$.  By \cite[Proposition 9.8.3]{KP}, there is a group scheme $\mc G_{x,r}$ of $G$ over $\O_K$ satisfying $G_{x,r}=\mc G_{x,r}(\O_K)$, and we write $\mf g_{x,r}=\Lie(\mc G_{x,r})(\O_K)$ for its Lie $\O_K$-algebra as in \cite[Section 13.3]{KP}.  We then have the following generalization of Proposition \ref{qsplit-LieAlg-iso}.

\begin{thm}\label{general-LieAlg-iso}
There is a natural isomorphism of Lie $\Z_p$-algebras
$$L_\omega(G_{x,\frac1{p-1}+})_{\frac1{p-1}+}\iso \mf g_{x,\frac1{p-1}+}$$
which respects the filtrations; that is, it induces for all $r>\frac1{p-1}$ an isomorphism
$$L_\omega(G_{x,\frac1{p-1}+})_r\iso \mf g_{x,r}.$$
\end{thm}

\begin{proof}
As in the proof of Theorem \ref{general-existence}, let $L/K$ be a finite unramified (hence Galois) extension such that $G_L$ is a quasisplit reductive $L$-group.  Then Proposition \ref{qsplit-LieAlg-iso} provides us with compatible isomorphisms
\begin{equation}\label{basechange}
\log(G(L)_{x,r})\xto{\sim} \mf g(L)_{x,r}
\end{equation}
of Lie $\Z_p$-algebras for any $r>\frac1{p-1}$.  Recall that, as an isomorphism of $\Z_p$-modules, this is a direct sum of isomorphisms defined on each Iwahori factor; moreover, we have concrete descriptions of these maps on the Iwahori factors given by Remarks \ref{cochars} and \ref{qsplit-expl-desc}.  One can immediately deduce from these descriptions that the map \ref{basechange} is in fact $\Gal(L/K)$-equivariant, from which we find an isomorphism of Lie $\Z_p$-algebras
$$\log\left(G(L)_{x,r}^{\Gal(L/K)}\right)\xto{\sim} \mf g(L)_{x,r}^{\Gal(L/K)},$$
which can be rewritten as $\log(G_{x,r})\xto{\sim} \mf g_{x,r}$.  Because these isomorphisms are compatible, it follows that we have an isomorphism $\log(G_{x,\frac1{p-1}+})\xto{\sim} \mf g_{x,\frac1{p-1}+}$ which respects the filtrations.
\end{proof}

From this result, we are also able to compute the Lie $\F_p[\eps]$-algebra $\gr G_{x,\frac1{p-1}+}$.  This is a more conceptual generalization of \cite[Theorem 3.4]{dottolehung}, which was a result that only applied to the split case where $p-1>he$ (see Remark \ref{splitcase} above); in such a case, $G_{x,\frac1{p-1}+}=G_{x,0+}$ is $I_1$ in their notation.  In fact, we are able to obtain information about the classical Moy-Prasad isomorphism (see \cite[Section 13.5]{KP}).

Recall from Proposition \ref{log-exp} that $\gr L_\omega(G_{x,\frac1{p-1}+})_{\frac1{p-1}+}$, as a Lie algebra over $\gr\Z_p\iso \F_p[\eps]$, is isomorphic to $\gr G_{x,\frac1{p-1}+}$.

\begin{cor} \label{MP-cor}
The Moy-Prasad isomorphisms
$$\mathrm{MP}_{x,r,r+}: G_{x,r}/G_{x,r+}\xto{\sim} \mf g_{x,r}/\mf g_{x,r+}$$
for $r>\frac1{p-1}$ fit together to form an isomorphism of Lie $\F_p[\eps]$-algebras
\begin{equation}\label{MP}\mathrm{MP}: \gr G_{x,\frac1{p-1}+}\xto{\sim} \gr \mf g_{x,\frac1{p-1}+}.\end{equation}
\end{cor}

To be precise, we view $\gr \mf g_{x,\frac1{p-1}+}$ as a Lie algebra over $\F_p[\eps]$ by first viewing $\mf g_{x,\frac1{p-1}+}$ as a Lie algebra over $\Z_p$ (it is a priori over $\O_K$), then take graded parts to arrive at an action of $\F_p[\eps]\iso \gr\Z_p$.  Thus, $\eps$ is a degree-one map which sends any $g+\mf g_{x,r}$ to $pg+\mf g_{x,r+1}$.

\begin{proof}
Note first that these Moy-Prasad isomorphisms exist by \cite[Theorem 13.5.1]{KP}, as our condition on the torus ensures that their condition (1) is satisfied.

It immediately follows from Theorem \ref{general-LieAlg-iso} and Proposition \ref{log-exp} that we have an isomorphism of Lie $\F_p[\eps]$-algebras 
$$\gr G_{x,\frac1{p-1}+}\xto{\gr\log} \gr L_\omega(G_{x,\frac1{p-1}+})_{\frac1{p-1}+}\xto{\sim} \gr \mf g_{x,\frac1{p-1}+},$$
where the second isomorphism is the one arising from Theorem \ref{general-LieAlg-iso} by taking graded parts.  All that remains to be proven is that this coincides with \ref{MP}.  Because both our construction and the one in \cite{KP} arise by first passing to a sufficiently large unramified extension $L/K$ such that $G_L$ is quasisplit and then taking Galois invariants, we may suppose without loss of generality that $G$ is quasisplit.  Similarly, both constructions proceed by building the isomorphisms on each Iwahori factor and then taking direct sums, so it is sufficient to show that our maps are equal on each Iwahori factor.  For the torus, they agree by Remark \ref{cochars} and the proof of \cite[Proposition B.6.9]{KP}.  For the root groups, they agree by the functoriality in \cite[Proposition A.5.19(3)]{KP} and Remark \ref{qsplit-expl-desc}.
\end{proof}

\newpage
\bibliographystyle{amsalpha}
\bibliography{pval}
\end{document}